\documentclass[12pt, reqno]{amsart}
\usepackage{amssymb}
\usepackage{graphicx} 
\usepackage{color}
\usepackage{url}
\usepackage{fancyref}
\usepackage{epstopdf}
\usepackage{amsfonts}
\usepackage{enumerate}
\usepackage[colorlinks = true,
            linkcolor = blue,
            urlcolor  = blue,
            citecolor = blue,
            anchorcolor = blue]{hyperref}
\usepackage{pgfplotstable}
\usepackage{array}
\usepackage{adjustbox}
\usepackage{subcaption}

\numberwithin{equation}{section}

\newtheorem{lemma}{Lemma}

\newtheorem{theorem}{Theorem}
\newtheorem{corollary}{Corollary}

\begin{document}

\baselineskip 18pt

\newcommand{\E}{\mathbb{E}}
\newcommand{\Eof}[1]{\mathbb{E}\left[ #1 \right]}
\newcommand{\Et}[1]{\mathbb{E}_t\left[ #1 \right]}
\renewcommand{\H}{\mathbb{H}}
\newcommand{\R}{\mathbb{R}}
\newcommand{\sigl}{\sigma_L}
\newcommand{\BS}{\rm BS}
\newcommand{\rv}{{\rm RV}}
\newcommand{\p}{\partial}
\renewcommand{\P}{\mathbb{P}}
\newcommand{\Pof}[1]{\mathbb{P}\left[ #1 \right]}
\newcommand{\var}{{\rm var}}
\newcommand{\cov}{{\rm cov}}
\newcommand{\beaa}{\begin{eqnarray*}}
\newcommand{\eeaa}{\end{eqnarray*}}
\newcommand{\bea}{\begin{eqnarray}}
\newcommand{\eea}{\end{eqnarray}}
\newcommand{\ben}{\begin{enumerate}}
\newcommand{\een}{\end{enumerate}}
\newcommand{\bit}{\begin{itemize}}
\newcommand{\eit}{\end{itemize}}

\newcommand{\bsb}{\boldsymbol{b}}
\newcommand{\bX}{\boldsymbol{X}}
\newcommand{\bx}{\boldsymbol{x}}
\newcommand{\by}{\boldsymbol{y}}
\newcommand{\bE}{\mathbf{e}}
\newcommand{\bw}{\mathbf{w}}
\newcommand{\bsR}{\boldsymbol{R}}
\newcommand{\bW}{\boldsymbol{W}}
\newcommand{\bB}{\boldsymbol{B}}
\newcommand{\bZ}{\boldsymbol{Z}}
\newcommand{\bH}{\mathbf{H}}
\newcommand{\bF}{\mathbf{F}}
\newcommand{\bG}{\mathbf{G}}
\newcommand{\bs}{\mathbf{s}}
\newcommand{\bt}{\mathbf{t}}
\newcommand{\bsihat}{\widehat{\bs_i}}
\newcommand{\bseta}{\boldsymbol{\eta}}

\newcommand{\sgn}{\mbox{sgn}}

\newcommand{\mt}{\mathbf{t}}
\newcommand{\mS}{\mathbb{S}}

\newcommand{\argmax}{{\rm argmax}}
\newcommand{\argmin}{{\rm argmin}}

\newcommand{\tM}{\widetilde{M}}
\newcommand{\tEof}[1]{\tilde{\mathbb{E}}\left[ #1 \right]}
\newcommand{\tP}{\tilde{\mathbb{P}}}
\newcommand{\tW}{\tilde{W}}
\newcommand{\tB}{\tilde{B}}
\newcommand{\1}{\mathbf{1}}
\renewcommand{\O}{\mathcal{O}}
\newcommand{\dt}{\Delta t}
\newcommand{\tr}{{\rm tr}}

\newcommand{\Xv}{X^{(v)}}
\newcommand{\Xvs}{X^{(v^*)}}
\newcommand{\Jv}{J^{(v)}}

\newcommand{\cA}{\mathcal{A}}
\newcommand{\cG}{\mathcal{G}}
\newcommand{\cF}{\mathcal{F}}
\newcommand{\cL}{\mathcal{L}}
\newcommand{\cLv}{\mathcal{L}^{(v)}}
\newcommand{\cV}{\mathcal{V}}

\newcommand{\ttheta}{\tilde{\theta}}
\newcommand{\vega}{{\rm vega}}

\newcommand{\inn}[1]{\langle #1 \rangle}

\newcommand{\cred}{\textcolor[rgb]{0.80,0.00,0.00}}

\allowdisplaybreaks





\title[Entropy-regularized robust optimal order execution]{Relative entropy-regularized robust optimal order execution}
\begin{abstract}
The problem of order execution is cast as a relative entropy-regularized robust optimal control problem in this article. 
The order execution agent's goal is to maximize an objective functional associated with his profit-and-loss of trading and simultaneously minimize the inventory risk associated with market's liquidity and uncertainty. We model the market's liquidity and uncertainty by the principle of least relative entropy associated with the market trading rate. 
The problem of order execution is made into a relative entropy-regularized stochastic differential game. Standard argument of dynamic programming yields that the value function of the differential game satisfies a relative entropy-regularized Hamilton-Jacobi-Isaacs (rHJI) equation. Under the assumptions of linear-quadratic model with Gaussian prior, the rHJI equation reduces to a system of Riccati and linear differential equations. Further imposing constancy of the corresponding coefficients, the system of differential equations can be solved in closed form, resulting in analytical expressions for optimal strategy and trajectory as well as the posterior distribution of market trading rate. Numerical examples illustrating the optimal strategies and the comparisons with conventional trading strategies are conducted. 
\end{abstract}


\author[M. Wang]{Meng Wang}
\author[T.-H. Wang]{Tai-Ho Wang}

\address{Meng Wang \newline
Financial Engineering, \newline
Baruch College, The City University of New York \newline
1 Bernard Baruch Way, New York, NY10010
}
\email{maggiewm.pku@gmail.com}

\address{Tai-Ho Wang \newline
Department of Mathematics \newline
Baruch College, The City University of New York \newline
1 Bernard Baruch Way, New York, NY10010
}
\email{tai-ho.wang@baruch.cuny.edu}

\keywords{Optimal execution, Price impact, Inventory cost, Entropy-regularized stochastic control}


\maketitle


%
%


\section{Introduction}
Order execution, a mission that algorithmic trading departments and execution brokerage agencies embark on regularly, is cast as a relative entropy-regularized robust optimal control problem. During the course of executing a large order of significant amount, the agent faces with not only the risk of price impact that his own execution would incur towards the transaction price but also the liquidity and uncertainty of the market. The agent's goal is to maximize an objective functional associated with his profit-and-loss (P\&L) of trading and simultaneously minimize the execution risk. 

A quote by Kyle in \cite{kyle} states that
\begin{quote}
{\it Roughly speaking, Black defines a liquid market as one which is almost infinitely tight, which is not infinitely deep, and which is resilient enough so that prices eventually tend to their underlying value.}
\end{quote}
As such, we model the market's liquidity and uncertainty by the {\it principle of least relative entropy} associated with the market trading rate. In other words, the market is resilient enough so that the probability distribution of trading rate of the market, though impacted due to the presence of the order execution agent's trading, would stay as close, in terms of the Kullback-Leibler divergence, as possible to the distribution prior to the presence of the agent's trades. Such consideration also provides a framework for the order execution agent to assess the market liquidity and uncertainty risk incurred from his own trade. Henceforth, the agent will be able to dynamically adjust his trading strategy accordingly. 
It worths to mention that, in literatures on optimal order execution, market resilience was usually modeled as the resilience of limit order book. For instance, as a pioneer work in this direction, \cite{obizhaeva-wang} considered the case of flat order book density and  \cite{alfonsi-schied} generalized it to general order book shape. In these models, the resilience of limit order book is depicted, after the limit order book is hit or lifted by market orders, as the recovery in an exponential rate back to its original stationary form of order book density. In the current article, we take a different route, we model market resilience by the principle of least relative entropy as mentioned earlier. This formulation also has the advantage of applying reinforcement learning techniques, similar to \cite{kim-yang} and \cite{wzz}, to the optimal execution framework. The regularization by Kullback-Leibler divergence prevents drastic change to market trading rate, the market is thus considered resilient in this sense. 

The order execution agent is considered as playing a stochastic differential game against the market in that he attempts to minimize the entropy-regularized risk while at the same time maximizing his own P\&L at the liquidation horizon, with possible penalty at terminal time should there be a final block trade. Standard argument of dynamic programming remains applicable in this setting, it follows that the value function of the differential game satisfies a relative entropy-regularized Hamilton-Jacobi-Isaacs (rHJI) equation. 
Under the assumptions of linear-quadratic model with Gaussian prior, the rHJI equation reduces to a system of Riccati and linear differential equations. Further imposing constancy of the corresponding coefficients, the system of differential equations can be solved in closed form, resulting in analytical expressions for the optimal trading strategy, the optimal expected trading trajectory as well as the posterior distribution of market trading rate. The optimal strategies obtained in this case resemble the renowned Almgren-Chriss strategy in \cite{almgren-chriss}. However, parts of the parameters involved in the determination of the strategies are from the prior distribution of market trading rate, the liquidity and uncertainty risk, as well as relative entropy penalty factor, which are naturally embedded into the optimal strategies. 

The rest of the paper is organized as follows. We formulate and model the order execution problem in Section \ref{sec:model-setup}. Section \ref{sec:execution-game}  recasts the optimal order execution problem as a relative-entropy regularized stochastic differential game and derives the associated Hamilton-Jacobi-Isaacs equation. Section \ref{sec:soln} proposes two model assumptions that allow us to solve the Hamilton-Jacobi-Isaacs equation in closed form and presents the associated optimal trading strategies and trajectories. 
Verification theorems and duality between the primal and the dual problems are shown in Sections \ref{sec:verification-duality}. Section \ref{sec:non_constant_volatility} briefly discusses the extension to non-constant price volatility.   
Numerical illustrations and discussions on the results are reported in Section \ref{sec:numerical}. 
Section \ref{sec:conclusion} concludes the paper with discussions. Finally, we briefly review the relative entropy-regularized control problem and its corresponding Hamilton-Jacobi-Bellman equation, called the {\it relative entropy-regularized HJB equation}, as an appendix in Section \ref{sec:appendix}.  

Throughout the paper, $\left(\Omega, \mathcal{F}, \mathbb{P}\right)$ denotes a complete probability space equipped with a filtration describing the information structure $\mathbb{F} := \left\{\mathcal{F}_{t}\right\}_{t \in [0,T]}$, where $t$ is the time variable and $T>0$ the fixed finite liquidation horizon. Let $\left\{W^S_t, W^X_t\right\}_{t \in [0,T]}$ be a two-dimensional Brownian motion with constant correlation $\rho$ defined on $\left(\Omega, \mathcal{F}, \mathbb{P}\right)$. The filtration $\mathbb{F}$  is generated by the trajectories of the above Brownian motion, completed with all $\mathbb{P}$-null measure sets of $\mathcal{F}$.


\section{Model setup} \label{sec:model-setup}

In this section, we layout the problem and set up the model in a continuous time setting for an order execution agent who is missioned to liquidate a given amount of shares of certain stock under price impact within the time interval $[0, T]$. In addition to execution risk, the model takes into account market's liquidity and uncertainty risk. In the following, $X_t$ denotes the agent's holdings at time $t$, $S_t$ the efficient or mid price of the stock, and $\tilde S_t$ the agent's transacted price at time $t$. 

\subsection{Price impact model}

To take into account the liquidity and uncertainty risk from market trading rate when executing orders, we propose a price impact model as follows. Let $a_t$ be the random market trading rate at time $t$, following distribution $\pi_t$. 
The efficient price $S_t$ is assumed governed by the stochastic differential equation (SDE)
\begin{equation}
dS_t = \gamma dX_t + \gamma_M \inn{a}_t dt + \sigma_t^S dW_t^S, \label{eqn:S}
\end{equation}
where $\inn{\cdot}_t$ denotes averaging over $\pi_t$, i.e., $\inn{a}_t = \int a \pi_t(a) da$ and $\sigma_t^S = \sqrt{\int \tilde\sigma^2(a) \pi_t(a) da}$ for some $\tilde\sigma$. Though it is well-documented that the volatility $\sigma^S_t$ is highly correlated with market trading rate, we shall for simplicity assume $\sigma^S$ is constant throughout the article. 
The traded price $\tilde S_t$ at time $t$ is given by 
\begin{equation}
\tilde S_t = S_t + \eta v_t
\end{equation}
for $\eta > 0$, where $v_t$ is the agent's intended trading rate at time $t$. The $\gamma dX_t$ terms in \eqref{eqn:S} is usually referred to as the {\it permanent impact} and the $\eta v_t$ term in $\tilde S_t$ as the {\it temporary impact}. Note that in \eqref{eqn:S} we assume the averaged market trading rate $\inn{a}$ also has an impact permanently to the efficient price; however, it carries no temporary impact contributing to the transacted price. This assumption is consistent with regular activities in the market in the sense that if in average there are more buy orders than sell orders in the market, i.e., $\inn{a} > 0$, the efficient price will increase; otherwise, the price decreases. Also if buy and sell orders even out in average, apart from the possible permanent impact incurred from the agent's trading, the efficient price is a martingale. 

We remark that, to our knowledge, most literatures on order execution do not particularly single out the market trading rate $\inn{a}$ but rather include it in the diffusion part. However, there exist literatures considering impact from other market participants by imposing stochastic price impact. For example, \cite{mariani-fatone} introduces an additional noise to the execution price resulting from other investors' trading activities. \cite{ma-siu-zhu-elliott} models the net permanent impact from other participants in the market by a mean-reverting process, and both permanent and temporary price impact coefficients of the investor are assumed to follow a jump process on a finite state space. In this paper, we factor it out from the diffusion term so as to assess the market's liquidity and uncertainty risk in order execution. We then consider the agent's order execution problem as a robust control problem, i.e., optimizing the objective function in the worst case scenario as will be described in Section \ref{sec:execution-game}.

\subsection{Agent's P\&L}
A recent paper by Carmona and Leal in \cite{carmona-leal} showed pretty convincingly that the presence of a Brownian component in the agent's inventory is statistically significant. To incorporate, we model the agent's holdings $X_t$ at time $t$ as
\begin{eqnarray}
&& dX_t = v_t dt + \sigma^X dW_t^X ,\quad t\in [0,T], \label{eqn:X} \\
&& X_0 = x_{0}>0 \nonumber 
\end{eqnarray}
should the agent's intended trading rate is $v_t$ at time $t$.
The Brownian motion component $W_t^X$ and the volatility $\sigma^X$ proxy the uncertainty of the agent's holdings. The two Brownian motions $W_t^S$ in \eqref{eqn:S} and $W_t^X$ in \eqref{eqn:X} are assumed correlated with correlation $\rho$. Concerning the correlation $\rho$ between the two Brownian motions, it is argued in \cite{carmona-webster} that the correlation $\rho$ is negative when
trading with market orders and positive when trading with limit orders.
Also, optimal order execution in this situation has been studied. For example, 
\cite{c-dg-w1} focuses on the order execution problem with different utility functions, \cite{c-dg-w2} in addition considers the evolution of the riskiness of the penalty due to the final block trade, \cite{dg-t-w} introduces the price pressure driven by market makers' inventories' risk, and \cite{mariani-fatone} adds additional noise in the execution price due to the activities of other investors.

The agent's profit-and-loss (P\&L hereafter) $\Pi^0_t$ during execution course up to time $t$ is given by 
\begin{equation}
\Pi^{0}_t := X_t \left( S_t - S_0 \right) + \int_{0}^{t} (S_0 - \tilde{S}_u) dX_u. \label{eqn:pl0}
\end{equation}
Note that the first term on the right-hand side of \eqref{eqn:pl0} captures the change in fair (marked-to-market) value of the remaining untransacted shares, while the second term measures the transaction gain/loss resulting from selling shares due to spread, volatility, and price impact. Moreover, since the agent is mandated to liquidate all his position at the liquidation horizon $T$, if he is left in his holdings with $X_T$ shares at the terminal time $T$, his final P\&L $\Pi_T$ at time $T$ is to be penalized by $g(X_T)$, where $g(0) = 0$ and $g < 0$ if $X_T \neq 0$. Thus, the agent's P\&L $\Pi_T$ after taking a final block trade at time $T$ is given by
\[
\Pi_T = \Pi^0_T + g(X_T).
\]

We note that, by applying Ito's formula and straightforward calculations, the P\&L $\Pi^0_T$ may be rewritten, without explicit dependence on price $S$, as  
\begin{eqnarray}
\Pi^0_T 
&=& \frac\gamma2 (X_T^2 - X_0^2) + \gamma_M \int_0^T X_t \inn{a}_t dt + \int_0^T X_t \sigma^S dW^S_t + \frac\gamma2 \int_0^T (\sigma^X)^2 dt  \nonumber \\
&& + \int_0^T \rho\sigma^S \sigma^X dt - \eta \int_0^T v_t^2 dt - \eta \int_0^T v_t \sigma^X dW_t^X.  \label{eqn:pnl0-1}
\end{eqnarray}
Hence, the expected P\&L $\Pi^0$ prior to a final block trade is given by 
\begin{eqnarray*}
\Eof{\Pi^0_T} &=& -\frac\gamma2 X_0^2 + \Eof{\frac\gamma2 X_T^2 + \int_0^T \left\{\gamma_M X_t \inn{a}_t + \frac\gamma2  (\sigma^X)^2 + \rho\sigma^S \sigma^X - \eta v_t^2 \right\} dt}. 
\end{eqnarray*}
We shall ignore the constant term $-\frac\gamma2 X_0^2$ in the equation above when time comes to formulate and solve the agent's order execution as a control problem.  

\subsection{Execution and market risk}
The agent's risk, which consists of the execution risk and the market's liquidity and uncertainty risk within the time interval $[t, T]$, is modeled as 
\begin{eqnarray*}
g_M(S_T) + \int_t^T R(u, X_u, S_u, v_u, a_u) du
\end{eqnarray*}
for some given terminal risk $g_M$ on the stock price and running risk $R$. For example, the running risk $R$ may include the quadratic variation of P\&L $\Pi_0$ over liquidation horizon as one of its risk components. We refer to the assumptions imposed in Section \ref{sec:soln} for more detailed discussions on imposing the running risk $R$. 

In summary, the performance criterion for the agent's execution is given by the execution P\&L penalized by the execution risk and the market liquidity and uncertainty risk. Namely, 
\[
g(X_T) + X_t \left( S_t - S_0 \right) + \int_{0}^{t} (S_0 - \tilde{S}_u) dX_u + g_M(S_T) + \int_t^T R(u, X_u, S_u, v_u, a_u) du.
\]
We recall that the first three terms in the equation above correspond to the execution P\&L penalized by a final block trade while the last two terms represent risks.  

\section{Order execution as regularized stochastic differential game} \label{sec:execution-game}
With the model set up in Section \ref{sec:model-setup}, we recast the agent's order execution as a {\it relative entropy-regularized stochastic differential game} in this section. The agent is regarded as playing a game against the market in a conservative manner in the sense that, rather than rigorously minimizing the risks incurred from the market as in the methodology of  robust optimal control, he only attempts to minimize entropy-regularized risk while at the same time maximizing his P\&L at the liquidation horizon. 

Alternatively, inspired by Fisher Black's notion on liquid market as stated in \cite{kyle}, which we reiterate in the following
\begin{quote}
{\it ``Roughly speaking, Black defines a liquid market as one which is almost infinitely tight, which is not infinitely deep, and which is resilient enough so that prices eventually tend to their underlying value''},
\end{quote} 
this concept of minimizing the relative entropy-regularized risk can be thought of as representing the market resilience via the {\it principle of least relative entropy} for market activity. In other words, the distribution $\pi$ of market's trading rate $a$ is also impacted due to the presence of the agent's trading. However, it is resilient enough so as to stay as close, in terms of the Kullback-Leibler divergence, as possible to the prior distribution $\pi^0$. 

\subsection{Order execution problem}
We regard the differential game played by the market is to minimize the following relative entropy-regularized control problem among admissible distributions $\pi\in\cA$ for the market trading rate $a$
\begin{eqnarray*}
\min_{\pi\in\cA} \Eof{g_M(S_T) + \int_0^T \int\left\{R(u, X_u, S_u, v_u, a) + \frac1\beta\log\frac{\pi_u(a)}{\pi^0_u(a)} \right\} \pi_u(a)da du}
\end{eqnarray*}
for some $\beta > 0$. The set of admissible distributions $\cA$ consists of the distributions $\pi_t$ that are adapted to the filtration $\cF_t$ and absolutely continuous to a given prior distribution $\pi^0_t$ for $t \in [0, T]$. We now formulate the agent's order execution problem as the following relative entropy-regularized robust optimal control problem. 
\begin{eqnarray}
&& \max_{v\in\cV} \min_{\pi\in\cA} \E\left[g(X_T) + X_T S_T  - \int_0^T \tilde S_t dX_t \right. \nonumber \\
&& \qquad + g_M(S_T) + \left. \int_0^T \int\left\{R(t, X_t, S_t, v_t, a) + \frac1\beta\log\frac{\pi_t(a)}{\pi^0_t(a)} \right\} \pi_t(a)da\,dt \right], \label{eqn:rr-control}
\end{eqnarray}
where $\cV$ consists of adapted, square integrable processes in $[0, T]$.  

We remark that (a) the collection of admissible distributions in \eqref{eqn:rr-control} represents a Knightian uncertainty\footnote{Knightian uncertainty is named after the renowned economist Frank Knight who formalized a distinction between risk and uncertainty in his 1921 book entitled {\it Risk, Uncertainty, and Profit}. An easy to access explanation on Knightian uncertainty can be found at the webpage
\href{https://news.mit.edu/2010/explained-knightian-0602}{news.mit.edu/2010/explained-knightian-0602}.}; (b) the inner minimization problem can be further subject to certain constraints on the posterior distribution such as first and second order statistics. Finally, by substituting \eqref{eqn:pnl0-1} for P\&L, the problem \eqref{eqn:rr-control} is transformed into 
\begin{eqnarray}
&& \max_{v\in\cV} \min_{\pi\in\cA} \Eof{G(X_T, S_T) \right.  \label{eqn:rr-control1} \\
&& \qquad \quad \left. 
+ \int_0^T\!\!\!\!\int\left\{\tilde R(t, X_t, S_t, v_t, a) + \frac1\beta\log\frac{\pi_t(a)}{\pi^0_t(a)} \right\} \pi_t(a)da\,dt}, \nonumber
\end{eqnarray}
where $G(X_T, S_T) := g(X_T) + \frac\gamma2 X_T^2 + g_M(S_T)$ and 
\[
\tilde R(t, x, s, v, a) = \gamma_M x a + \frac\gamma2  (\sigma^X)^2 + \rho\sigma^S \sigma^X - \eta v^2 + R(t, x, s, v, a).
\]
We shall be mainly dealing with the problem \eqref{eqn:rr-control1} in the following sections.  

\subsection{The relative entropy-regularized Hamilton-Jacobi-Isaacs equation}
To solve the execution problem \eqref{eqn:rr-control1} which can be regarded as a stochastic differential game, define the value function $V$ by 
\begin{eqnarray}
&& V(t, x, s) \label{eqn:value_func} \\ 
&=& \!\!\!\! \max_{v\in\cV_t} \min_{\pi\in\cA_t} \E_t\left[G(X_T ,S_T) 
+ \int_t^T\!\!\!\!\!\int\left\{\tilde R(\tau, X_\tau, S_\tau, v_\tau, a_\tau) + \frac1\beta\log\frac{\pi_\tau(a)}{\pi^0_\tau(a)} \right\} \pi_\tau(a)da\,d\tau\right]. \nonumber
\end{eqnarray}
where $\E_t[\cdot]$ denotes the conditional expectation $\Eof{\cdot|\cF_t}$ and assume $(X_t, S_t) = (x, s)$.
$\cV_t$ and $\cA_t$ denote the admissible strategies and distributions in $\cV$ and $\cA$ respectively that are restricted to the time interval $[t, T]$.
By applying the standard dynamical programming principle argument, one can show that the value function $V$ in \eqref{eqn:value_func} satisfies a {\it relative entropy-regularized Hamilton-Jacobi-Isaacs} (rHJI hereafter) equation which we summarize in the following theorem whose proof is omitted. 
\begin{theorem}
The value function $V$ in \eqref{eqn:value_func} satisfies the following rHJI equation
\begin{equation} \label{eqn:rHJI}
\max_{v} \min_{\pi \ll \pi_t^0} \int \left\{V_t + \cL V + \tilde R(t, x, s, v, a) + 
\frac1\beta\log\frac{\pi(a)}{\pi^0_t(a)} \right\} \pi(a) da = 0, 
\end{equation}
with terminal condition $V(T, x, s) = G(x, s)$, assuming enough regularity for the value function $V$. $\cL$ denotes the infinitesimal generator for the SDEs
\begin{eqnarray*}
&& dS_t = \gamma dX_t + \gamma_M a_t dt + \sigma^S dW^S_t, \\
&& dX_t = v_t dt + \sigma^X dW^X_t,
\end{eqnarray*}
where the Brownian motions $W_t^S$ and $W^X_t$ are correlated with correlation $\rho$.
\end{theorem}

\section{Solutions in closed form} \label{sec:soln}
In this section, we show under further model assumptions that the rHJI equation \eqref{eqn:rHJI} can be reduced into a Hamilton-Jacobi equation and present, in the linear-quadratic (LQ) framework, solutions in closed form to the resulting effective Hamilton-Jacobi equation. We remark that in the following when describing a Gaussian distribution, we shall mostly use the term {\it precision}, defined by the reciprocal of variance, rather than the commonly used term of variance since it simplifies the notations a little. 

The following assumptions are imposed in order for the problem \eqref{eqn:rr-control1} to remain in a Gaussian and LQ structure though other choices are by all means possible.  
\bit
\item[{\bf A1.}] The ``{\it prior}" distribution $\pi^0_t$ is Gaussian with mean $m_t$ and precision $s_t$. Precisely, 
\[
\pi^0_t(a) = \sqrt{\frac{s_t}{2\pi}} e^{-\frac{s_t}2(a - m_t)^2}.
\]
\item[{\bf A2.}] $g_M(s) = 0$, $g(x) = -\delta x^2$. $R$ is independent of $s$ and is quadratic in $x$, $a$. Specifically,  
\[
R = R(t, x, a) = \frac{R_{xx}}2 x^2 + R_{xa} xa + \frac{R_{aa}}2 a^2 
\] 
where the coefficients may be time dependent. Moreover, the coefficient $R_{aa}$ is further assumed satisfying $s_t + \beta R_{aa} > 0$ for all $t$, where recall that $s_t$ is the precision for the prior Gaussian distribution $\pi_t^0$ given in Assumption {\bf A1}.
\item[{\bf A3.}] \label{A3} $g_M(s) = 0$, $g(x) = -\delta x^2$. $R$ is independent of $s$ and is quadratic in $v$, $a$. Specifically,  
\[
R = R(t, v, a) = \frac{R_{vv}}2 v^2 + R_{va} va + \frac{R_{aa}}2 a^2 
\] 
where the coefficients may be time dependent. Moreover, the coefficients $R_{vv}$, $R_{va}$ , and $R_{aa}$ are further assumed satisfying
\beaa
s_t + \beta R_{aa} > 0 \quad \mbox{and} \quad \eta - \frac{R_{vv}}{2} + \frac{\beta R_{va}^2}{2(s_t + \beta R_{aa})} > 0
\eeaa
for all $t$.

\eit

We briefly explain the financial rationale for the assumptions {\bf A2} and {\bf A3} on $R$ being quadratic.
Concerning {\bf A2}, as one observes in \eqref{eqn:pnl0-1} that the $x^2$ term may result from the quadratic variation of P\&L $\Pi^0_t$. For the $xa$ term, consider the following {\it volume-weighted averaged efficient price} for the trading trajectory $X_t$.
\begin{eqnarray}
&& \frac{1}{X_0}\int_0^T S_t dX_t 
= \frac{1}{X_0}\left\{S_TX_T - X_0 S_0 - \int_0^T X_t dS_t - [X, S]_T \right\} \nonumber \\
&\approx& -S_0 + \frac{1}{X_0}\left\{ - \gamma \int_0^T X_t dX_t - \gamma_M \int_0^T X_t a_t dt - \int_0^T X_t \sigma^S dW^S_t - \int_0^T \rho\sigma^S \sigma^X dt \right\} \nonumber \\
&\approx& -S_0 + \frac{1}{X_0}\left\{ - \gamma \int_0^T X_t dX_t - \gamma_M \int_0^T X_t a_t dt - \int_0^T X_t \sigma^S dW^S_t\right\} \nonumber \\
&\approx& -S_0 + \frac\gamma2 X_0 - \frac{\gamma_M}{X_0} \int_0^T X_t a_t dt - \frac1X_0 \int_0^T X_t \sigma^S dW^S_t. \label{eqn:vwaep}
\end{eqnarray} 
Hence, apart from the constants and the martingale term in \eqref{eqn:vwaep}, if the agent is subject to minimize the volume-weighted averaged efficient price of his strategy, he may consider penalizing his final P\&L by a factor of the third term in \eqref{eqn:vwaep}, which is a special case of {\bf A2}.   
As for {\bf A3}, consider the situation where the order execution agent is required to track a percentage-of-volume (POV) strategy as a benchmark. That is, with a given market participant rate $\nu > 0$, at any point in time he would like his trading $v_t$ to stay close to $\nu a_t$, should the market trading rate be $a_t$. In this case, he may penalize his P\&L by the quantity $\lambda |v - \nu a|^2$, where $\lambda > 0$ is a risk aversion parameter, which indeed is a special case for $R$ in {\bf A3}.

In both {\bf A2} and {\bf A3}, $R$ is assumed independent of price $S_t$. The rationale is that risk, be it of price or of execution, is considered closely related with the relative price movements, which originate from the agent's and the market's trading activities, rather than the absolute price level. We note that, since under these assumptions the problem \eqref{eqn:rr-control1} and hence its associated value function do not depend on $S_t$ any more, the SDE for $S_t$ can be discharged from the problem \eqref{eqn:rr-control1} and thus the infinitesimal generator $\cL$ in the rHJI equation \eqref{eqn:rHJI} reduces to $\cL = \frac{(\sigma^X)^2}2 \p_x^2 + v \p_x$. 

\subsection{Model 1} \label{sec:model-1}
In this subsection, we show that under Assumptions {\bf A1} and {\bf A2}, which we refer to as {\it Model 1}, the rHJI equation \eqref{eqn:rHJI} can be reduced to an effective Hamilton-Jacobi equation that admits a solution in semi-analytical form subject to solving a system of ODEs. Further assuming time independent coefficients, the solution to the aforementioned system of ODEs can be obtained in closed form, yielding closed form expressions for the optimal trading rate and the optimal posterior distribution for market trading rate. 

The following lemma is crucial in the calculations involved in reducing the rHJI equations into effective Hamilton-Jacobi equations and for the proof of duality in Section \ref{sec:duality}. 
\begin{lemma} \label{lma:crucial}
For a distribution $\pi$ satisfying $\pi \ll \pi_t^0$, consider the functional
\[
F[\pi] := \int \left\{ \frac{C_2}2 a^2 + C_1 a + \frac1\beta \log \frac{\pi(a)}{\pi_t^0(a)}\right\} \pi(a) da.
\]
where $C_1$ and $C_2$ are constants such that $s_t + \beta C_2 > 0$. 
Then, $F$ is convex in $\pi$ and its minimum is achieved uniquely at the distribution $\pi^*$ whose probability density function is given by 
\begin{equation}
\pi_t^*(a) = \sqrt{\frac{s_t + \beta C_2}{2\pi}}e^{-\frac{(s_t + \beta C_2)(a-w^*)^2}{2}}, \quad w^* = \frac{s_t m_t - \beta C_1}{s_t + \beta C_2},  \label{eqn:minimizer_pi}
\end{equation} 
i.e., $\pi^*$ is a normal density with mean $w^*$ in \eqref{eqn:minimizer_pi} and precision $s_t + \beta C_2$. The minimum value is given by
\begin{eqnarray}
&& F[\pi^*] = \frac{1}{2\beta} \log\frac{s_t + \beta C_2}{s_t} + \frac{s_t m_t^2}{2\beta} - \frac1{2\beta} \frac{(s_t m_t - \beta C_1)^2}{s_t + \beta C_2}.  \label{eqn:min_value}
\end{eqnarray}
\end{lemma}
\begin{proof}
We first show that $F$ is convex. For any given distributions $\pi_1, \pi_2 \ll \pi_t^0$ and $\lambda \in [0, 1]$, let $\pi = \lambda \pi_1 + (1 - \lambda)\pi_2$. We have
\begin{eqnarray*}
&& F[\pi] = F[\lambda \pi_1 + (1 - \lambda)\pi_2] \\
&=& \int \left\{ \frac{C_2}2 a^2 + C_1 a + \frac1\beta \log \frac{\pi(a)}{\pi_t^0(a)}\right\} \pi(a) da \\
&=& \int \left\{ \frac{C_2}2 a^2 + C_1 a \right\}\{\lambda \pi_1(a) + (1 - \lambda)\pi_2(a)\} da +  \int \frac1\beta \frac{\pi(a)}{\pi_t^0(a)}\log \frac{\pi(a)}{\pi_t^0(a)} \pi_t^0(a) da \\
&\leq& \lambda \int \left\{ \frac{C_2}2 a^2 + C_1 a \right\} \pi_1(a) da + (1 - \lambda) \int \left\{ \frac{C_2}2 a^2 + C_1 a \right\} \pi_2(a) da \\
&& + \lambda \int \frac1\beta \frac{\pi_1(a)}{\pi_t^0(a)}\log \frac{\pi_1(a)}{\pi_t^0(a)} \pi_t^0(a) da 
+ (1 - \lambda) \int \frac1\beta \frac{\pi_2(a)}{\pi_t^0(a)}\log \frac{\pi_2(a)}{\pi_t^0(a)} \pi_t^0(a) da \\
&=& \lambda F[\pi_1] + (1 - \lambda) F[\pi_2],
\end{eqnarray*}
where the inequality results from the fact that the function $x\log x$ is convex. We thus conclude that $F$ is convex. 

As for the minimum value, note that for any $\pi\ll\pi_t^0$ we have 
\begin{eqnarray*}
&& \int \left\{ \frac{C_2}{2}a^2 + C_1 a + \frac{1}{\beta}\log\frac{\pi(a)}{\pi_t^0 (a)}  \right\} \pi(a) da \\
&=& \int \left\{ \frac{C_2}{2}a^2 + C_1 a + \frac{1}{\beta}\log\frac{\pi^*(a)}{\pi_t^0 (a)} \right\} \pi(a) da + \frac1\beta \int \log\frac{\pi(a)}{\pi^* (a)} \pi(a) da \\
&\geq & \int \left\{ \frac{C_2}{2}a^2 + C_1 a + \frac{1}{\beta}\log\frac{\pi^*(a)}{\pi_t^0 (a)} \right\} \pi(a) da \\
&=& \left[C_1 + \frac{s_t + \beta C_2}{\beta}w^* - \frac{s_t m_t}{\beta} \right] \inn{a}_t 
 + \frac{1}{2\beta} \log\frac{s_t + \beta C_2}{s_t} + \frac{s_t m_t^2}{2\beta} - \frac{s_t + \beta C_2}{2\beta} (w^*)^2 \\
&=& \frac{1}{2\beta} \log\frac{s_t + \beta C_2}{s_t} + \frac{s_t m_t^2}{2\beta} - \frac1{2\beta} \frac{(s_t m_t - \beta C_1)^2}{s_t + \beta C_2} \\
&=& \int \left\{ \frac{C_2}{2}a^2 + C_1 a + \frac{1}{\beta}\log\frac{\pi_t^*(a)}{\pi_t^0 (a)}  \right\} \pi_t^*(a) da,
\end{eqnarray*}
where the inequality results from the positivity of relative entropy. 
Thus, the minimum value of is achieved at $\pi^*$ in \eqref{eqn:minimizer_pi} and the minimum value is given by \eqref{eqn:min_value}.
\end{proof}

\begin{lemma} \label{lma:inner-min-value}
In Model 1, the minimal value for the inner minimization in \eqref{eqn:rHJI} is equal to  
\begin{eqnarray}
&& \!\!\! V_t + \frac{(\sigma^X)^2}2 V_{xx} + v V_x + \frac\gamma2(\sigma^X)^2 + \rho\sigma^S \sigma^X - \eta v^2 -\frac1{2\beta} \log \left(\frac{s_t}{s_t + \beta R_{aa}}\right) 
\label{eqn:inner-min-value} \\
&+& \frac12\left( R_{xx} - \frac{\beta (\gamma_M + R_{xa})^2}{s_t+\beta R_{aa}}\right) x^2 + \frac{(\gamma_M + R_{xa})s_t m_t}{s_t+\beta R_{aa}} x 
+ \frac{s_t m_t^2 R_{aa}}{2(s_t+\beta R_{aa})}. \nonumber
\end{eqnarray}
\end{lemma}
\begin{proof}
For any given trading rate $v_t$ and function $V$, the first order criterion via variational calculus applied to the inner minimization in \eqref{eqn:rHJI}, see also \eqref{eqn:optimal-pi} and \eqref{eqn:rHJB-reduced} in the appendix, yields that the candidate minimizer $\pi_t^*$ for the minimization in the rHJI equation \eqref{eqn:rHJI} is given by 
\begin{equation}
\pi_t^*(a) = \sqrt{\frac{s_t + \beta R_{aa}}{2\pi}}e^{-\frac{(s_t + \beta R_{aa})(a-w^*)^2}{2}}, \quad w^* = \frac{-\beta (\gamma_M + R_{xa}) x + s_t m_t}{s_t + \beta R_{aa}},  \label{eqn:optimal_pi}
\end{equation} 
since $s_t + \beta R_{aa} > 0$, and its corresponding value is equal to
\begin{equation}
-\frac1\beta \log \int \pi^0_t(a|x) e^{-\beta \left[V_t + \cL V + \tilde R(t, x, a) \right]} da. \label{eqn:rrHJI}
\end{equation}

We evaluate the integral in \eqref{eqn:rrHJI} as follows. 
\begin{eqnarray*}
&& \int \pi^0_t(a|x) e^{-\beta \left\{V_t + \cL V + \tilde R(t, x, a) \right\}} da \\
&=& \int e^{-\beta\left\{V_t + \cL V + \frac\gamma2(\sigma^X)^2 + \rho\sigma^S \sigma^X - \eta v^2 + \frac12 R_{xx} x^2 + (\gamma_M + R_{xa}) xa + \frac12 R_{aa} a^2 \right\}} \sqrt{\frac{s_t}{2\pi}}e^{-\frac{s_t}2 (a-m_t)^2} da \\
&=& e^{-\beta\left\{V_t + \frac{(\sigma^X)^2}2 V_{xx} + v V_x + \frac\gamma2(\sigma^X)^2 + \rho\sigma^S \sigma^X - \eta v^2 + \frac12 R_{xx} x^2\right\}} \\
&& \quad \times \int e^{-\beta\left\{(\gamma_M + R_{xa}) xa + \frac12 R_{aa} a^2 \right\}} \sqrt{\frac{s_t}{2\pi}}e^{-\frac{s_t}2 (a-m_t)^2} da \\
&=& e^{-\beta\left\{V_t + \frac{(\sigma^X)^2}2 V_{xx} + v V_x + \frac\gamma2(\sigma^X)^2 + \rho\sigma^S \sigma^X - \eta v^2 \right\}} \\
&& \quad \times \sqrt{\frac{s_t}{s_t + \beta R_{aa}}} e^{-\frac\beta2\left\{R_{xx} + \frac{-\beta (\gamma_M + R_{xa})^2}{s_t+\beta R_{aa}}\right\} x^2 +\frac{-\beta (\gamma_M + R_{xa})s_t m_t}{s_t+\beta R_{aa}} x - \frac{\beta s_t m_t^2 R_{aa}}{2(s_t+\beta R_{aa})} }.
\end{eqnarray*}
Hence, by taking logarithm of the last equation and multiplying the resulting expression by the factor $-\frac1\beta$, we get the desired quantity.

The $\pi_t^*$ in \eqref{eqn:optimal_pi} is indeed a minimizer for the minimization in \eqref{eqn:rHJI} since, for any admissible distribution $\pi_t$, the objective function, after omitting all the terms independent of $a$, equals
\begin{eqnarray}\nonumber
&& \int \left\{ (\gamma_M + R_{xa})xa + \frac{R_{aa}}{2}a^2 + \frac{1}{\beta}\log\frac{\pi_t(a)}{\pi_t^0 (a)}  \right\} \pi_t(a) da.
\end{eqnarray}
Hence, with $C_2 = R_{aa}$ and $C_1 = (\gamma_M + R_{xa})x$, Lemma \ref{lma:crucial} implies that $\pi_t^*$ is the unique minimizer. 
\end{proof}

Next, we move on to deal with the outer maximization in \eqref{eqn:rHJI} and show that it can be further reduced to the Hamilton-Jacobi equation \eqref{eqn:HJ}. 
\begin{lemma} \label{lma:HJ-1}
The rHJI equation \eqref{eqn:rHJI} is reduced to the following Hamilton-Jacobi equation
\begin{eqnarray}
&& V_t + \frac{(\sigma^X)^2}2 V_{xx} + \frac{V_x^2}{4\eta}
- \frac1{2\beta} \log \left(\frac{s_t}{s_t + \beta R_{aa}}\right) + \frac\gamma2(\sigma^X)^2 + \rho\sigma^S \sigma^X  \label{eqn:HJ} \\
&& \qquad + \frac12\left\{ R_{xx} - \frac{\beta (\gamma_M + R_{xa})^2}{s_t+\beta R_{aa}}\right\} x^2 + \frac{(\gamma_M + R_{xa})s_t m_t}{s_t+\beta R_{aa}} x \nonumber \\
&& \qquad + \frac{s_t m_t^2 R_{aa}}{2(s_t+\beta R_{aa})} = 0 \nonumber
\end{eqnarray}
with terminal condition $V(T, x) = -g x^2$, where $g = \delta - \frac\gamma2$.
\end{lemma} 
\begin{proof}
Lemma \ref{lma:inner-min-value} showed that the value of the inner minimization in \eqref{eqn:rHJI} is equal to \eqref{eqn:inner-min-value}. Thus, \eqref{eqn:rHJI} reduces to 
\begin{eqnarray}
&& \max_v \left\{ V_t + \frac{(\sigma^X)^2}2 V_{xx} + v V_x + \frac\gamma2(\sigma^X)^2 + \rho\sigma^S \sigma^X - \eta v^2 -\frac1{2\beta} \log \left(\frac{s_t}{s_t + \beta R_{aa}}\right) \right. \nonumber \\
&& \qquad + \frac12\left( R_{xx} - \frac{\beta (\gamma_M + R_{xa})^2}{s_t+\beta R_{aa}}\right) x^2 + \frac{(\gamma_M + R_{xa})s_t m_t}{s_t+\beta R_{aa}} x \nonumber \\
&& \left. \qquad + \frac{s_t m_t^2 R_{aa}}{2(s_t+\beta R_{aa})}  \right\} = 0. \label{eqn:outer-max}
\end{eqnarray}
The expression between the brackets is quadratic and concave in $v$, it admits a unique maximal value at the vertex. The first order criterion thus implies that the maximizer $v^*$ is given by $v^* = \frac{V_x}{2\eta}$. The proof is then completed by plugging the maximizer $v^*$ into \eqref{eqn:outer-max}.
\end{proof}

\begin{theorem} \label{thm:ODE1}
In Model 1, the value function $V$ for the control problem \eqref{eqn:rr-control1} is quadratic in $x$, say, $V(t, x) = \frac12 H_2 x^2 + H_1 x + H_0$, where the time dependent coefficients $H_2$, $H_1$, and $H_0$ satisfy the following system of equations
\begin{eqnarray*}
&& \dot H_2 + \frac{H_2^2}{2\eta} + A_1 = 0, \\
&& \dot H_1 + \frac{H_1H_2}{2\eta} + B_1 = 0, \\
&& \dot H_0 + \frac{(\sigma^X)^2}2 H_2 + \frac{H_1^2}{4\eta} + \frac\gamma2 (\sigma^X)^2 + \rho\sigma^S \sigma^X - \frac1{2\beta}\log c - \frac{s_t}{2\beta} m_t^2 (c - 1) = 0
\end{eqnarray*}
with terminal conditions $H_2(T) = -2g$ and $H_1(T) = H_0(T) = 0$. $A_1$, $B_1$, and $c$ are defined by
\begin{eqnarray*}
&& A_1 = R_{xx} -\frac{\beta (\gamma_M + R_{xa})^2}{s_t+\beta R_{aa}}, \qquad
B_1 = \frac{(\gamma_M + R_{xa})s_t m_t}{s_t+\beta R_{aa}}, \quad
c = \frac{s_t}{s_t + \beta R_{aa}}.
\end{eqnarray*}
\end{theorem}
\begin{proof}
Assume the ansatz for value function
\[
V(t, x) = \frac12 H_2 x^2 + H_1 x + H_0, 
\]
where the $H_i$'s are functions of $t$. Substitute the ansatz for value function to the Hamilton-Jacobi equation \eqref{eqn:HJ} then comparing the coefficients yield the following system of ODEs satisfied by the $H_i$'s
\begin{eqnarray*}
x^2 &:& \dot H_2 + \frac{H_2^2}{2\eta} + A_1 = 0, \\
x &:& \dot H_1 + \frac{H_1H_2}{2\eta} + B_1 = 0, \\
1 &:& \dot H_0 + \frac{(\sigma^X)^2}2 H_2 + \frac{H_1^2}{4\eta} + \frac\gamma2 (\sigma^X)^2 + \rho\sigma^S \sigma^X - \frac1{2\beta}\log c - \frac{s_t}{2\beta} m_t^2 (c - 1) = 0.
\end{eqnarray*}
\end{proof}

We note that, since the coefficients in the system of ODEs in Theorem \ref{thm:ODE1} may be time dependent, in general it does not admit solution in closed form. However, if we further assume the coefficients being constant, the system of ODEs does admit a solution in closed form. We summarize the result in the following corollary.
\begin{corollary}
Further assume that all the coefficients in the system of ODEs in Theorem \ref{thm:ODE1} are constant and let 
\begin{eqnarray*}
&& \hat A_1 = \sqrt{\frac{-A_1}{2\eta}},
\end{eqnarray*}
assuming $A_1 < 0$. The functions $H_2$ and $H_1$ have the following closed form expressions
\begin{eqnarray*}
&& H_2(t) = -2\eta \hat A_1 \coth\left( \hat A_1\{T - t\} + \alpha_1 \right), \\
&& H_1(t) = - \frac{B_1}{\hat A_1} \frac{\cosh\alpha_1}{\sinh\left( \hat A_1\{T - t\} + \alpha_1 \right)} + \frac{B_1}{\hat A_1} \coth\left( \hat A_1\{T - t\} + \alpha_1 \right),
\end{eqnarray*}
where $\alpha_1 = \coth^{-1}\left(\frac{g}{\eta \hat A_1}\right)$. It follows that the optimal trading strategy $v_t^*$ is obtained by
\begin{eqnarray}
v_t^* &=& \frac1{2\eta}(H_2 x_t + H_1) \label{eqn:vst_feedback} \\
&=& -\hat A_1 \coth\left( \hat A_1\{T - t\} + \alpha_1\right) x_t  - \frac{B_1}{2\eta\hat A_1}\frac{\cosh(\alpha_1) - \cosh(\hat A_1\{T - t\} + \alpha_1)}{\sinh\left( \hat A_1\{T - t\} + \alpha_1 \right)}  \nonumber
\end{eqnarray} 
and the expected optimal trading trajectory $x_t^*$ by 
\begin{eqnarray}
x^*_t &=& \frac{\sinh\left(\hat A_1\{T - t\} + \alpha_1 \right)}{\sinh\left( \hat A_1 T + \alpha_1 \right)} x_0 \label{eqn:optimal-x-1} \\
&& \quad + \sinh\left( \hat A_1\{T - t\} + \alpha_1 \right) \int_0^t \frac{H_1(s)}{2\eta \sinh\left( \hat A_1\{T - s\} + \alpha_1 \right)} ds. \nonumber
\end{eqnarray}
\end{corollary}
We remark that, if the prior mean $m_t$ is identically zero, then $B_1 = 0$. It follows that the function $H_1$ is also identically zero. In this case, the expected optimal trajectory in \eqref{eqn:optimal-x-1} reduces to 
\begin{equation}
x^*_t = \frac{\sinh\left( \hat A_1\{T - t\} + \alpha_1 \right)}{\sinh\left( \hat A_1 T + \alpha_1 \right)} X  \label{eqn:optimal-x-m1-mean0}
\end{equation}
which resembles the Almgren-Chriss strategy except that the proxy of market trading rate precision $s_t$ is naturally embedded in the parameter $\hat A_1$. Moreover, the limiting behavior of $v_t^*$ as $g\rightarrow \infty$ and $A_1\rightarrow 0$ is obtained in the following corollary.
\begin{corollary}
Assume $m_t=0$. As $g\rightarrow \infty$ and $A_1\rightarrow 0$, the optimal execution strategy in feedback from \eqref{eqn:vst_feedback} converges to the adapted TWAP strategy defined as
\begin{equation} 
v_t^*=-\frac{x_t}{T-t}.  \label{eqn:adap-twap}
\end{equation}
A sufficient condition for $g\rightarrow\infty$ is $\delta\rightarrow\infty$, which means the final penalty on the inventory is extremely high. A sufficient condition for $A_1\rightarrow 0$ is $R_{xx} \rightarrow 0$ and $s_t\rightarrow \infty$, which corresponds respectively to zero quadratic term in $x$ for risk and the prior distribution being a Dirac measure concentrated at $m_t$.
\end{corollary}
\begin{proof}
If $m_t=0$, then $B_1=0$. It thus suffices to show that
\begin{equation*}
\lim_{g\rightarrow\infty,\hat{A}_1\rightarrow 0} \hat{A}_1 \coth\left(\hat{A}_1\{T-t\}+\alpha_1  \right) = \frac{1}{T-t}.
\end{equation*}
Indeed,
\begin{equation*}
\begin{aligned}
& \lim_{g\rightarrow\infty,\hat{A}_1\rightarrow 0} \hat{A}_1 \coth\left(\hat{A}_1\{T-t\}+\alpha_1  \right) \\
=& \lim_{g\rightarrow\infty,\hat{A}_1\rightarrow 0} \hat{A}_1\frac{1+\coth\left(\hat{A}_1\{T-t\} \right) \frac{g}{\eta\hat{A}_1}}{\coth\left(\hat{A}_1\{T-t\} \right) +\frac{g}{\eta\hat{A}_1}} \\
=& \lim_{g\rightarrow\infty,\hat{A}_1\rightarrow 0} \frac{\frac{\eta\hat{A}_1^2}{g} + \hat{A}_1 \coth\left(\hat{A}_1\{T-t\} \right) }{\frac{\eta}{g} \hat{A}_1 \coth\left(\hat{A}_1\{T-t\} \right) + 1}.
\end{aligned}
\end{equation*}
Since
\begin{equation*}
\lim_{\hat{A}_1\rightarrow 0} \hat{A}_1 \coth\left(\hat{A}_1\{T-t\} \right) =\frac{1}{T-t},
\end{equation*}
we reach the conclusion.
\end{proof}

\subsection{Model 2} \label{sec:model-2}
In this subsection, as for Model 1 in Section \ref{sec:model-1}, we show that under Assumptions {\bf A1} and {\bf A3}, which we shall refer to as {\it Model 2} hereafter, the rHJI equation \eqref{eqn:rHJI} is reduced to the effective Hamilton-Jacobi equation \eqref{eqn:HJ-2} that admits a solution in semi-analytical form subject to solving a system of ODEs. We also obtain closed form expressions for the solution in the case of time independent coefficients, yielding closed form expressions for the optimal trading rate and the optimal posterior distribution for market trading rate.

The following lemma shows how to reduce the rHJI equation \eqref{eqn:rHJI} to the effective Hamilton-Jacobi equation \eqref{eqn:HJ-2}. 
\begin{lemma}
In Model 2, the rHJI equation \eqref{eqn:rHJI} is reduced to the following effective Hamilton-Jacobi equation.
\begin{eqnarray}
&& V_t + \frac{(\sigma^X)^2}2 V_{xx} + \frac14\frac{\left(V_x + \frac{R_{va} s_t m_t}{s_t + \beta R_{aa}} - \frac{\beta \gamma_M R_{va}}{s_t + \beta R_{aa}} x \right)^2}{\eta - \frac{R_{vv}}2 + \frac{\beta R_{va}^2}{2(s_t + \beta R_{aa})}} \label{eqn:HJ-2} \\
&& \quad \qquad - \frac{\beta \gamma_M^2}{2(s_t + \beta R_{aa})} x^2 + \frac{s_t}{2\beta} m_t^2 \left(1 - \frac{s_t}{s_t + \beta R_{aa}}\right) - \frac1{2\beta} \log \left(\frac{s_t}{s_t + \beta R_{aa}}\right) \nonumber \\
&& \quad \qquad + \frac{\gamma_M s_t m_t}{s_t + \beta R_{aa}} x + \frac\gamma2(\sigma^X)^2 + \rho\sigma^S \sigma^X = 0. \nonumber
\end{eqnarray}
with terminal condition $V(T, x) = -g x^2$, where $g = \delta - \frac\gamma2$.
\end{lemma} 
\begin{proof}
Since the proof of the lemma is almost parallel to that of Lemma \ref{lma:inner-min-value} and Lemma \ref{lma:HJ-1} for Model 1, we show only the essential calculations for deriving the effective Hamilton-Jacobi equation \eqref{eqn:HJ-2}. The first order criterion via variational calculus applied to the inner minimization in \eqref{eqn:rHJI} yields that the minimizer $\pi_t^*$ for the inner minimization in the rHJI equation \eqref{eqn:rHJI} is given by 
\[
\pi_t^*(a) = \sqrt{\frac{s_t + \beta R_{aa}}{2\pi}}e^{-\frac{s_t + \beta R_{aa}}{2} (a-w^*)^2}, \quad w^* = \frac{-\beta R_{va} v - \beta \gamma_M x + s_t m_t}{s_t + \beta R_{aa}}, 
\]
since $s_t + \beta R_{aa} > 0$, and its corresponding value by 
\begin{equation}
-\frac1\beta \log \int \pi^0_t(a|x) e^{-\beta \left[V_t + \cL V + \tilde R(t, x, s, v, a) \right]} da. \label{eqn:rrHJI-2}
\end{equation}
We now calculate the integral in \eqref{eqn:rrHJI-2} as follows. We have
\begin{eqnarray*}
&& \int \pi^0_t(a|x) e^{-\beta \left\{V_t + \cL V + \tilde R(t, x, s, v, a) \right\}} da \\
&=& \int e^{-\beta\left\{V_t + \cL V + \frac\gamma2(\sigma^X)^2 + \rho\sigma^S \sigma^X - \eta v^2 + \gamma_M xa + \frac12 R_{vv} v^2 + R_{va} va + \frac12 R_{aa} a^2 \right\}} \sqrt{\frac{s_t}{2\pi}}e^{-\frac{s_t}2 (a-m_t)^2} da \\
&=& e^{-\beta\left\{V_t + \frac{(\sigma^X)^2}2 V_{xx} + v V_x + \frac\gamma2(\sigma^X)^2 + \rho\sigma^S \sigma^X - \eta v^2 + \frac12 R_{vv} v^2\right\}} \\
&& \qquad \int e^{-\beta\left\{(\gamma_M x + R_{va}v)a + \frac12 R_{aa} a^2 \right\}} \sqrt{\frac{s_t}{2\pi}}e^{-\frac{s_t}2 (a-m_t)^2} da \\
&=& \!\!\! e^{-\beta\left\{V_t + \frac{(\sigma^X)^2}2 V_{xx} + v V_x + \frac\gamma2(\sigma^X)^2 + \rho\sigma^S \sigma^X - \eta v^2 + \frac12 R_{vv}v^2 \right\}} \sqrt{\frac{s_t}{s_t + \beta R_{aa}}} e^{\frac{\left(s_t m_t - \beta\{\gamma_M x + R_{va} v\}\right)^2}{2(s_t + \beta R_{aa})} - \frac{s_t}2 m_t^2}.
\end{eqnarray*}
Next, by taking logarithm of the last equation and multiplying the resulting expression by the factor $-\frac1\beta$, then substitute into \eqref{eqn:rrHJI-2}, the rHJI equation reduces to the following HJB equation
\begin{eqnarray}
&& \max_v \left\{V_t + \frac{(\sigma^X)^2}2 V_{xx} + v V_x + \frac\gamma2(\sigma^X)^2 + \rho\sigma^S \sigma^X - \left(\eta - \frac{R_{vv}}2\right) v^2  \right. \label{eqn:hjb-2} \\
&& \left. \quad \qquad -\frac1{2\beta} \log \left(\frac{s_t}{s_t + \beta R_{aa}}\right) - \frac{\left(s_t m_t - \beta\{\gamma_M x + R_{va} v\}\right)^2}{2\beta(s_t + \beta R_{aa})} 
+ \frac{s_t}{2\beta} m_t^2 \right\} = 0. \nonumber
\end{eqnarray}
%
%

The maximization problem in \eqref{eqn:hjb-2}, after omitting the terms that are independent of $v$, reads
\begin{eqnarray*}
\max_{v\in\mathbb{R}}\left\{ -\left(\eta - \frac{R_{vv}}{2}\right) v^2 + v V_x - \frac{\left[ \beta R_{va} v + \beta \gamma_M x - s_t m_t \right]^2}{2\beta (s_t + \beta R_{aa})} \right\} 
\end{eqnarray*}
Since the objective function in the maximization problem above is quadratic in $v$, its unique maximum is attended at
\begin{equation}
v = \frac12 \frac{V_x + \frac{R_{va} s_t m_t}{s_t + \beta R_{aa}} - \frac{\beta \gamma_M R_{va}}{s_t + \beta R_{aa}} x }{\eta - \frac{R_{vv}}{2} + \frac{\beta R_{va}^2}{2(s_t + \beta R_{aa})}} \label{eqn:optimizer-2}
\end{equation} 
since $\eta - \frac{R_{vv}}{2} + \frac{\beta R_{va}^2}{2(s_t + \beta R_{aa})} > 0$ by assumption.

Finally, by substituting the maximizer \eqref{eqn:optimizer-2} into the HJB equation \eqref{eqn:hjb-2}, we obtain the effective Hamilton-Jacobi equation given in \eqref{eqn:HJ-2}, which completes the proof. 
\end{proof}

Following almost the same procedure as for Model 1 in Section \ref{sec:model-1}, one can show that the value function in this case is also quadratic in $x$ and the time dependent coefficients satisfy similar Riccati equations as in Theorem \ref{thm:ODE1}. We summarize the result in the following theorem but omit its proof.
\begin{theorem} \label{thm:ODE2}
In Model 2, the value function $V$ for the control problem \eqref{eqn:rr-control1} is quadratic in $x$, say, $V(t, x) = \frac12 H_2 x^2 + H_1 x + H_0$, where the time dependent functions $H_2$ and $H_1$ satisfy the following system of ODEs
\begin{eqnarray*}
&& \dot H_2 + \frac1{2\tilde\eta} \left(H_2 - \frac{\beta \gamma_M R_{va}}{s_t + \beta R_{aa}} \right)^2 - A_2 = 0, \\
&& \dot H_1 + \frac1{2\tilde\eta} \left(H_2 - \frac{\beta \gamma_M R_{va}}{s_t + \beta R_{aa}} \right)\left(H_1 + \frac{R_{va} s_t m_t}{s_t + \beta R_{aa}}\right) + B_2 = 0.
\end{eqnarray*}
with terminal condition $H_2(T) = -2g$ and $H_1(T) = 0$. $A_2$, $B_2$, and $\tilde\eta$ are given by
\begin{eqnarray*}
&& A_2 = \frac{\beta \gamma_M^2}{s_t+\beta R_{aa}}, \qquad
B_2 = \frac{\gamma_M s_t m_t}{s_t + \beta R_{aa}}, \quad
\tilde \eta = \eta - \frac{R_{vv}}2 + \frac{\beta R_{va}^2}{2(s_t + \beta R_{aa})},
\end{eqnarray*}
where $s_t + \beta R_{aa} > 0$ and $\tilde\eta > 0$ by assumption.
\end{theorem}
Finally, as in Model 1, if we further assume the coefficients are constant, we obtain closed form expressions for the functions $H_2$ and $H_1$, thus closed form expressions for optimal trading rate and expected optimal trading strategy as well. We summarize the result in the following corollary without proof.
\begin{corollary}
Assume $\tilde\eta > 0$ and that all the coefficients are constant. Let 
\begin{eqnarray*}
&& \hat A_2 = \sqrt{\frac{A_2}{2\tilde\eta}}.
\end{eqnarray*}
The functions $H_2$ and $H_1$ admit the following closed form expressions
\begin{eqnarray*}
&& H_2(t) = -2\tilde\eta \hat A_2 \coth\left( \hat A_2\{T - t\} + \alpha_2 \right) + C, \\
&& H_1(t) = - \frac{D \hat A_2\sinh\alpha_2 + B_2\cosh\alpha_2}{\hat A_2\sinh\left( \hat A_2\{T - t\} + \alpha_2 \right)} + \frac{B_2}{\hat A_2} \coth\left( \hat A_2\{T - t\} + \alpha_2 \right) + D,
\end{eqnarray*}
where $\alpha_2 = \coth^{-1}\left(\frac{2g + C}{2\tilde\eta\hat A_2}\right)$, 
$C = \frac{\beta \gamma_M R_{va}}{s_t + \beta R_{aa}}$, and
$D = - \frac{R_{va} s_t m_t}{s_t + \beta R_{aa}}$. 
It follows that the optimal trading strategy $v_t^*$ is obtained by
\begin{eqnarray*}
v_t^* &=& \frac1{2\tilde\eta}(H_2 x_t + H_1-C x_t -D) \\
&=&  -\hat A_2 \coth\left( \hat A_2\{T - t\} + \alpha_2 \right)  x_t \\
&& \quad  - \frac{D \hat A_2 \sinh\alpha_2 + B_2\cosh\alpha_2}{2\tilde\eta\hat A_2\sinh\left( \hat A_2\{T - t\} + \alpha_2 \right)} + \frac{B_2}{2\tilde\eta\hat A_2} \coth\left( \hat A_2\{T - t\} + \alpha_2 \right). 
\end{eqnarray*} 
\end{corollary}
We remark that, if the prior mean $m_t$ is zero, then $B_2 = D = 0$, the optimal trading rate $v_t^*$ reduces to 
\begin{eqnarray}
v_t^* &=&  -\hat A_2 \coth\left( \hat A_2\{T - t\} + \alpha_2 \right) x_t. \label{eqn:optimal-v-m2-mean0}
\end{eqnarray}
In this case, the expected optimal trading trajectory $x_t^*$ reads
\begin{equation}
x_t^* =  \frac{\sinh\left( \hat A_2\{T - t\} + \alpha_2 \right)}{\sinh\left( \hat A_2 T + \alpha_2 \right)} X, \label{eqn:optimal-x-m2-mean0}
\end{equation} 
which again resembles the Almgren-Chriss strategy but the proxy of market trading rate precision $s_t$ is naturally embedded in the parameter $\hat A_2$. Finally, as in {\bf Model 1}, we state the limiting behavior of $v_t^*$ as as $g\rightarrow \infty$ and $A_2\rightarrow 0$ but omit its proof.
\begin{corollary}
Assume $m_t=0$, the optimal execution strategy converges to the adapted TWAP strategy \eqref{eqn:adap-twap} as $g\rightarrow \infty$ and $A_2\rightarrow 0$.
A sufficient condition for $g\rightarrow\infty$ is $\delta\rightarrow\infty$, which means a final block trade at terminal time is strictly prohibited. A sufficient condition for $A_2\rightarrow 0$ is $s_t\rightarrow \infty$, which corresponds to the prior distribution being a Dirac measure concentrated at $m_t$.
\end{corollary}

\section{Verification theorem and duality} \label{sec:verification-duality}

\subsection{Verification theorem} \label{sec:verification}
We prove the verification theorems for the relative entropy-regularized stochastic differential game \eqref{eqn:rr-control1} in this subsection. 
The argument is standard, relies mainly on the saddle point property for a max-min problem. In fact, the argument applies not only to our models but to general problems, subject to certain regularity conditions. 
\begin{theorem} \label{thm:verification}
Let $V(t, x, s)$ be the unique smooth solution to the following rHJI equation
\begin{equation} \label{rHJI_maxmin}
\max_v \min_{\pi\ll\pi_t^0} \int \left\{V_t + \cL V + \tilde R(t, x, s, v, a) + 
\frac1\beta\log\frac{\pi(a)}{\pi^0_t(a)} \right\} \pi(a) da = 0
\end{equation}
with terminal condition $V(T, x, s) = G(x, s)$.
Then $V$ is equal to the value function $U$ of the following maximization-minimization problem
\begin{eqnarray}
U(t, x, s) &=& \max_{v\in\cV_t} \min_{\pi \in\cA_t} \E_t\left[ G(X_T, S_T) \right. \nonumber \\
&& \qquad \left. + \int_t^T\!\!\!\!\int\left\{\tilde R(\tau, X_\tau, S_\tau, v_\tau, a_\tau) + \frac1\beta\log\frac{\pi_\tau(a)}{\pi^0_\tau(a)} \right\} \pi_\tau(a)da\,d\tau\right], \nonumber
\end{eqnarray}
where recall that $\E_t[\cdot]$ denotes the conditional expectation $\Eof{\cdot|\cF_t}$ and $(X_t, S_t) = (x,s)$.
%
\end{theorem}
\begin{proof}
For a given $t \in (0, T)$, consider the time interval $(t, T)$. Since $V$ satisfies the max-min equation \eqref{rHJI_maxmin}, we have that, for any given admissible control $v \in \cV_t$ in the interval $[t, T]$, there exists an admissible $\pi\in\cA_t$ such that
\[  
\int \left\{V_t + \cL V + \tilde R(\tau, x, s, v_\tau, a) + \frac1\beta\log\frac{\pi_\tau(a)}{\pi^0_\tau(a)} \right\} \pi_\tau(a) da \leq 0
\]
for every $\tau \in (t, T)$. Applying Ito's formula to $V$ then taking conditional expectation $\E_t[
\cdot]$ yields
\begin{eqnarray*}
&& \E_t\left[G(X_T, S_T) - V(t, X_t, S_t) \right] 
= \E_t\left[ V(T, X_T, S_T)\right] - V(t, x, s) \\
&=& \int_t^T \E_t\left[ V_{\tau} + \cL V \right] d\tau \\
&\leq& - \int_t^T  \E_t\left[\int \left\{\tilde R(\tau, X_{\tau}, S_{\tau}, v_{\tau}, a_{\tau}) + \frac1\beta\log\frac{\pi_{\tau}(a)}{\pi^0_{\tau}(a)} \right\} \pi_{\tau}(a) da \right] d\tau.
\end{eqnarray*}
It follows that
\begin{eqnarray*}
&& V(t,x,s) \\
&\geq& \E_t\left[ G(X_T, S_T) 
+ \int_t^T\!\!\!\!\int\left\{\tilde R(\tau, X_\tau, S_\tau, v_\tau, a_\tau) + \frac1\beta\log\frac{\pi_\tau(a)}{\pi^0_\tau(a)} \right\} \pi_\tau(a)da\,d\tau \right] \\
&\geq& \min_{\pi\in\cA_t} \E_t\left[ G(X_T, S_T) 
+ \int_t^T\!\!\!\!\int\left\{\tilde R(\tau, X_\tau, S_\tau, v_\tau, a_\tau) + \frac1\beta\log\frac{\pi_\tau(a)}{\pi^0_\tau(a)} \right\} \pi_\tau(a)da\, d\tau \right]
\end{eqnarray*}
Since $v$ is arbitrary, we end up 
\begin{eqnarray*}
&& V(t,x,s) \\
&\geq& \max_{v\in\cV_t} \min_{\pi\in\cA_t} \E_t\left[ G(X_T, S_T) 
+ \int_t^T\!\!\!\!\!\int\left\{\tilde R(\tau, X_\tau, S_\tau, v_\tau, a_\tau) + \frac1\beta\log\frac{\pi_\tau(a)}{\pi^0_\tau(a)} \right\} \pi_\tau(a)da\,d\tau\right] \\
&=& U(t, x, s).
\end{eqnarray*}

On the other hand, from the max-min equation \eqref{rHJI_maxmin}, we also have that, for any $\epsilon > 0$, there exists a $v \in \cV_t$ such that
\[  
\int \left\{V_t + \cL V + \tilde R(\tau, x, s, v_\tau, a) + \frac1\beta\log\frac{\pi_\tau(a)}{\pi^0_\tau(a)} \right\} \pi_\tau(a) da > -\epsilon.
\]
for $\tau \in (t, T)$ and all $\pi \in \cA_t$. Applying Ito's formula to $V$ and taking conditional expectation $\E_t[\cdot]$ yields
\begin{eqnarray*}
&& \E_t\left[ G(X_T, S_T) - V(t,X_t, S_t) \right] 
= \E_t\left[V(T, X_T, S_T)\right] - V(t, x, s) \\
&=& \int_t^T \E_t\left[V_{\tau} + \cL V\right] d\tau \\
&\geq& - \int_t^T  \E_t\left[\int \left\{\tilde R(\tau, X_{\tau}, S_{\tau}, v_{\tau}, a_{\tau}) + \frac1\beta\log\frac{\pi_{\tau}(a)}{\pi^0_{\tau}(a)} \right\} \pi_{\tau}(a) da \right] d\tau - \epsilon(T - t)
\end{eqnarray*}
It follows that
\begin{eqnarray*}
&& V(t,x,s) \\
&\leq& \E_t\left[ G(X_T, S_T) 
+ \int_t^T\!\!\!\!\!\int\left\{\tilde R(\tau, X_\tau, S_\tau, v_\tau, a_\tau) + \frac1\beta\log\frac{\pi_\tau(a)}{\pi^0_\tau(a)} \right\} \pi_\tau(a)da\,d\tau\right] + \epsilon(T - t)
\end{eqnarray*}
for all $\pi \in \cA_t$. We end up 
\begin{eqnarray*}
&& V(t, x, s) \\
&\leq& \min_{\pi\in\cA_t} \E_t\left[ G(X_T, S_T) \right. \\
&& \qquad \left. + \int_t^T\!\!\!\!\!\int\left\{\tilde R(\tau, X_\tau, S_\tau, v_\tau, a_\tau) + \frac1\beta\log\frac{\pi_\tau(a)}{\pi^0_\tau(a)} \right\} \pi_\tau(a)da\,d\tau\right] + \epsilon(T - t) \\
&\leq& \max_{v\in\cV_t} \min_{\pi\in\cA_t} \E_t\left[G(X_T, S_T) \right. \\
&& \qquad \left. + \int_t^T\!\!\!\!\!\int\left\{\tilde R(\tau, X_\tau, S_\tau, v_\tau, a_\tau) + \frac1\beta\log\frac{\pi_\tau(a)}{\pi^0_\tau(a)} \right\} \pi_\tau(a)da\,d\tau\right] + \epsilon(T - t) \\
&=& U(t, x, s) + \epsilon(T - t).
\end{eqnarray*}
We conclude that, since $\epsilon$ is arbitrary, $V(t, x, s) \leq U(t, x, s)$.
This completes the proof. 
\end{proof}

Finally, we remark that, by the same token, one can also prove the verification theorem for the ``dual" problem, which we summarize in the following theorem without proof. 
\begin{theorem} \label{thm:verification-dual}
Let $V(t, x, s)$ be the unique smooth solution to the following rHJI equation
\begin{equation} \label{rHJI_minmax}
\min_{\pi\ll\pi_t^0} \max_v \int \left\{V_t + \cL V + \tilde R(t, x, s, v, a) + 
\frac1\beta\log\frac{\pi(a)}{\pi^0_t(a)} \right\} \pi(a) da = 0
\end{equation}
with terminal condition $V(T, x, s) = G(x, s)$.
Then $V$ is equal to the value function $U$ of the following minimization-maximization problem
\begin{eqnarray}
U(t, x, s) &=& \min_{\pi \in\cA_t} \max_{v\in\cV_t} \E_t\left[ G(X_T, S_T) \right. \label{eqn:dual} \\
&& \qquad \left. + \int_t^T\!\!\!\!\int\left\{\tilde R(\tau, X_\tau, S_\tau, v_\tau, a_\tau) + \frac1\beta\log\frac{\pi_\tau(a)}{\pi^0_\tau(a)} \right\} \pi_\tau(a)da\,d\tau\right]. \nonumber
\end{eqnarray}
\end{theorem}
\subsection{Duality} \label{sec:duality}
The problem \eqref{eqn:rr-control1} is referred to as the primal problem and the problem \eqref{eqn:dual} as the dual problem. 
In this section, we prove that, for Model 1 and Model 2 considered respectively in Sections \ref{sec:model-1} and \ref{sec:model-2}, with an additional assumption for Model 2, the duality between the primal problem \eqref{eqn:rr-control1} and the dual problem \eqref{eqn:dual} holds. We hence conclude that, in the terminology of game theory, the stochastic differential game \eqref{eqn:rr-control1} is said to admit a value function as shown in \cite{fleming-souganidis} for the classical case. 

\begin{theorem} \label{thm:duality}
Under Mode1 1 and Model 2, for Model 2, further assume that $\eta > \frac{R_{vv}}{2}$, 
the primal problem \eqref{eqn:rr-control1} and its dual \eqref{eqn:dual} satisfy the strong duality, i.e., the two problems have the same value. 
\end{theorem}

\begin{proof}
It suffices to show that the Hamiltonian in the rHJI equations for the primal \eqref{eqn:rr-control1} and the dual \eqref{eqn:dual} satisfies the saddle point property. 
Recall that the Hamiltonian $H_1$ in Model 1 for the primal and the dual, after disregarding all the terms that are not dependent on $v$ or $a$ and $\pi$, is given by 
\[
H_1(t, x, v, \pi) := \!\!\!\int \left\{ -\eta v^2 + v V_x + \left(R_{xa} + \gamma_M \right) xa + \frac{R_{aa}}{2} a^2 + \frac1\beta\log\frac{\pi_t(a)}{\pi^0_t(a)} \right\} \pi_t(a) da 
\]
for $v \in \R$ and $\pi\ll\pi_t^0$.  
Notice that the $v$ dependence and the $(a, \pi)$ dependence in $H_1$ are separated.  
It follows that the Hamiltonian in Model 1 satisfies automatically the saddle point property 
\[
\max_v \min_{\pi\ll\pi_t^0} H_1(t, x, v, \pi) = \min_{\pi\ll\pi_t^0} \max_v H_1(t, x, v, \pi)
\]
for fixed $t$ and $x$.

As for Model 2. Note that, again by disregarding all the terms that are not dependent on $v$ or on $a$ and $\pi$, the Hamiltonian $H_2$ in Model 2 for the primal and the dual is given, for $v\in\R$ and $\pi\ll\pi_t^0$, by 
\begin{equation}
H_2(t, x, v, \pi) := \!\!\! \int \left\{-\hat\eta v^2 + v V_x + R_{va} va + \frac{R_{aa}}{2} a^2 + \gamma_M x a + \frac1\beta\log\frac{\pi_t(a)}{\pi^0_t(a)} \right\} \pi_t(a) da
\end{equation}
where $\hat \eta := \eta - \frac{R_{vv}}{2}$ and $\hat\eta > 0$ by assumption. 
Apparently, for any fixed $t, x, \pi$, $H_2$ is concave in $v$ and convex in $\pi$ (by Lemma \ref{lma:crucial}) for any fixed $t, x, v$, thus $H_2$ also satisfies the saddle point property 
\[
\max_v \min_{\pi\ll\pi_t^0} H_2(t, x, v, \pi) = \min_{\pi\ll\pi_t^0} \max_v H_2(t, x, v, \pi)
\]
for fixed $t$ and $x$. This completes the proof. 
\end{proof}
The duality theorem asserts that, under certain assumptions on model parameters, the relative entropy-regularized robust order execution problem, when regarded as a stochastic differential game, admits an equilibrium between the agent's optimal trading strategy and the market's optimal reaction per the principle of least relative entropy.

\section{Extensions to non-constant volatility} \label{sec:non_constant_volatility}
We consider a possible extension of the model to include non-constant, but only time dependent, volatility of the price dynamics in this section. Rather than being a constant, let the price volatility $\sigma^S$ be a function of time defined by 
\begin{equation}
\sigma_t^S = \sqrt{\int \Tilde{\sigma}^2(a)\pi_t(a) da}. \label{eqn:sigma_tS}
\end{equation}
In other words, the price volatility $\sigma^S$ is given by the square root of the average, in density $\pi_t$, of the variance of market trading rate $\tilde\sigma^2$. 
We summarize the findings in the following theorem without proof.
\begin{theorem}
The value function $V$ satisfies the following rHJI equation
\begin{equation*}
\max_{v} \min_{\pi \ll \pi_t^0} \left\{V_t + \cL V +\int \left\{ \hat R(t, x, s, v, a) + 
\frac1\beta\log\frac{\pi(a)}{\pi^0_t(a)} \right\} \pi(a) da + \rho\sigma^X \sigma_t^S \right\} = 0, 
\end{equation*}
with terminal condition $V(T, x, s) =  G(x, s)$, assuming enough regularity for the value function $V$, where
\begin{equation*}
\hat R(t, x, s, v, a) = \gamma_M xa + \frac{\gamma}{2}(\sigma^X)^2 - \eta v^2 + R(t,x,s,v,a).
\end{equation*}
$\cL$ denotes the infinitesimal generator for the SDEs
\begin{eqnarray*}
&& dS_t = \gamma dX_t + \gamma_M \langle a \rangle_t dt + \sigma_t^S dW^S_t, \\
&& dX_t = v_t dt + \sigma^X dW^X_t,
\end{eqnarray*}
where the Brownian motions $W_t^S$ and $W^X_t$ are correlated with correlation $\rho$.\\
When $\rho=0$ and under the assumptions of {\bf Model 1} or {\bf Model 2}, the optimal solution can be solved explicitly and it coincides with the optimal solution with constant volatility.
\end{theorem}

\section{Numerical examples} \label{sec:numerical}
In this section, we conduct numerical experiments on the implementation of optimal strategies obtained in Section \ref{sec:soln} and stress testing the strategies against various parameters. 
For Models 1 and 2, Monte Carlo simulations are conducted to illustrate the performances of the optimal strategies versus those of related adapted TWAP strategies as in \eqref{eqn:adap-twap}. We note that a TWAP strategy is meant to dice a meta order into child orders and gradually release them to the market using evenly divided time slots between the initial and terminal time. It follows that, in the continuous time limit, a TWAP strategy converges to trading in a constant rate. Its goal is to make execution price stay close to the time-averaged price between the initial and terminal time so as to minimize market impact of execution. Similarly, a VWAP (volume-weighted average price) strategy is a TWAP strategy except the evolution of time is in terms of cumulative total traded volume. However, as opposed to TWAP strategies which basically require zero modeling for market environment, implementation of a VWAP strategy requires the modeling of market volume process. In practice, TWAP and VWAP strategies are benchmarks for assessing performance of trading strategies or algorithms. 

The performance criterion $V$ in each experiment is decomposed into three additive components $V_{PnL}$, $V_{risk}$, and $V_{entropy}$. 
These components represent respectively contributions to the performance criterion $V$ from the expected P\&L, the risk, and the relative entropy. The performance criterion, as well as its three components, are then stress tested against certain extreme parameters. Monte Carlo simulations in each experiment are conducted with 4096 sample paths in 1000 time steps. The parameters chosen in the following numerical examples, mainly referenced to the parameters considered in \cite{almgren-chriss}, \cite{c-dg-w1}, and \cite{dg-t-w}, are for convenience only. In reality, the parameters need to be calibrated to the market data. 
\subsection{Model 1}
In Model 1, the three components for performance criterion $V$ are specified as 
\[
V = V_{PnL} + V_{risk} + V_{entropy}, 
\]
where
\beaa
&& V_{PnL} = -\delta X_T^2 + X_T S_T - X_0 S_0 - \int_0^T \tilde{S}_t dX_t, \\
&& V_{risk} = \int_0^T \left [ \frac{1}{2} R_{xx} x_t^2 + R_{xa} x_t \langle a \rangle_t + \frac{1}{2} R_{aa} \left(\langle a \rangle_t^2 + \frac{1}{s_t + \beta R_{aa}}  \right) \right] dt, \\
&& V_{entropy} = \int_0^T \left[\frac{1}{2\beta} \log\frac{s_t + \beta R_{aa}}{s_t} - \frac{R_{aa}}{2} \left( \langle a \rangle_t^2 + \frac{1}{s_t + \beta R_{aa}}  \right) + \frac{s_t + \beta R_{aa}}{ 2\beta} \langle a \rangle_t^2  \right] dt.
\eeaa

The strategies, with their corresponding optimal posterior mean of market trading rate, implemented are
\begin{enumerate}
\item Optimal strategy 
\beaa
&& v_t = -\hat{A}_1\coth\left\{ \hat{A}_1 (T-t) + \alpha_1 \right\} x_t, \qquad 
\langle a \rangle_t = -\frac{\beta (\gamma_M + R_{xa})}{s_t + \beta R_{aa}} x_t;
\eeaa

\item TWAP strategy
\beaa
&& v_t = -\frac{x_t}{T-t}, \qquad
\langle a \rangle_t = -\frac{\beta (\gamma_M + R_{xa})}{s_t + \beta R_{aa}}x_t.
\eeaa
\end{enumerate}
In the simulations that follow, parameters below are chosen and fixed across simulations. 
\beaa 
&& \gamma = 2.5\times 10^{-7}, \; R_{xa} = -5\times 10^{-6}, \; R_{aa} = 9 \times 10^{-7}, \; \beta=1, m_t = 0, \\
&& s_t = 10^{-8}, \; X=10^6, \; S=100, \; \sigma^X = 10^5, \; \sigma^S = 10, \; \rho=0.3, \; T=1.
\eeaa
The simulation results are shown in Figures \ref{model1_value_large_gammaM} through \ref{model1_value_large_Rxx}. In each figure, histograms for performance criterion and its three components are exhibited for the optimal and its related TWAP strategies; on top of the histograms show the Box plots generated by the data. 
The parameters as benchmark are set by
\[
\gamma_M = 2.5\times 10^{-6}, \; \eta = 2.5\times 10^{-6}, \;  \delta = 1.25\times 10^{-4}, R_{xx} = -10^{-6}
\]
whose simulation results are shown in Figure \ref{model1_value_benchmark}. 
Figure \ref{model1_value_large_gammaM} shows the result of stress testing the optimal and TWAP strategies against large permanent impact from market $\gamma_M = 10^{-5}$. Figure \ref{model1_value_large_eta_small_delta} shows the result of stress testing against relatively large $\eta=10^{-3}$ and small $\delta=2\times10^{-4}$. Finally, Figure \ref{model1_value_large_Rxx} is for stress testing against relatively large $R_{xx}=10^{-4}$. Table \ref{tab:table1} summarizes the parameters in each case and their corresponding figures.
\begin{table}[h!]
\begin{center}
    \caption{Parameters and figures for Model 1.}
    \label{tab:table1}
\begin{tabular}{|c|c|c|c|c|} 
\hline
     & $\gamma_M$ & $\eta$ & $\delta$ & $R_{xx}$  \\
\hline \hline
Benchmark & $2.5\times 10^{-6}$ & $2.5\times 10^{-6}$ & $1.25\times 10^{-4}$ & $-10^{-6}$ \\
\hline
Figure \ref{model1_value_large_gammaM} & $10^{-5}$ & $2.5\times 10^{-6}$ & $1.25\times 10^{-4}$ & $-10^{-6}$ \\
\hline
Figure \ref{model1_value_large_eta_small_delta} & $2.5\times 10^{-6}$ & $10^{-3}$ & $2\times 10^{-4}$ & $-10^{-6}$ \\
\hline
Figure \ref{model1_value_large_Rxx} & $2.5\times10^{-6}$ & $2.5\times 10^{-6}$ & $1.25\times 10^{-4}$ & $-10^{-4}$ \\
\hline
\end{tabular}
\end{center}
\end{table}

We observe from the simulation results that the total performances under TWAP and under optimal strategy in the benchmark case are almost identical while TWAP performs better in entropy but worse in risk. The variabilities of the parameters $\gamma_M$, $\gamma$, $\eta$, and $\delta$ mainly contribute to the variance in P\&L, while $R_{xx}$ mainly to variance of risk. We remark that in all the experiments, the magnitude of the entropy term $V_{entrop}$ are insignificant compared to those of the risk $V_{risk}$ and of the P\&L $V_{PnL}$. Also, it seems in Figure \ref{model1_value_large_eta_small_delta}, TWAP is doing better than the optimal strategy in $V_{risk}$. However, while the magnitude between $V_{PnL}$ and $V_{risk}$ in other cases are comparable, the magnitude of $V_{PnL}$ is almost a hundred times higher than that of $V_{risk}$ in this case. It turns out that in this case, in comparison with TWAP, the optimal strategy chooses to sacrifice risk for higher P\&L. Overall, apart from the benchmark case, the optimal strategies outperform and incurring lower variance as opposed to their TWAP counterparties as shown in the figures.

\begin{figure}
\centering
\includegraphics[width=.8\linewidth]{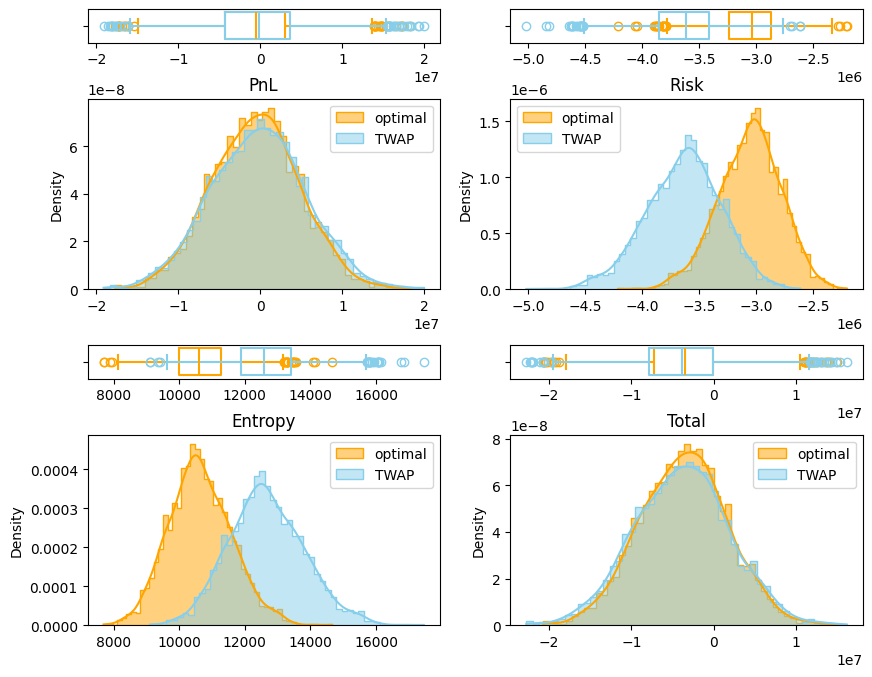}
\caption{Performance under benchmark parameters in Model 1.}
\label{model1_value_benchmark}
\end{figure}

\begin{figure}
\centering
\includegraphics[width=.8\linewidth]{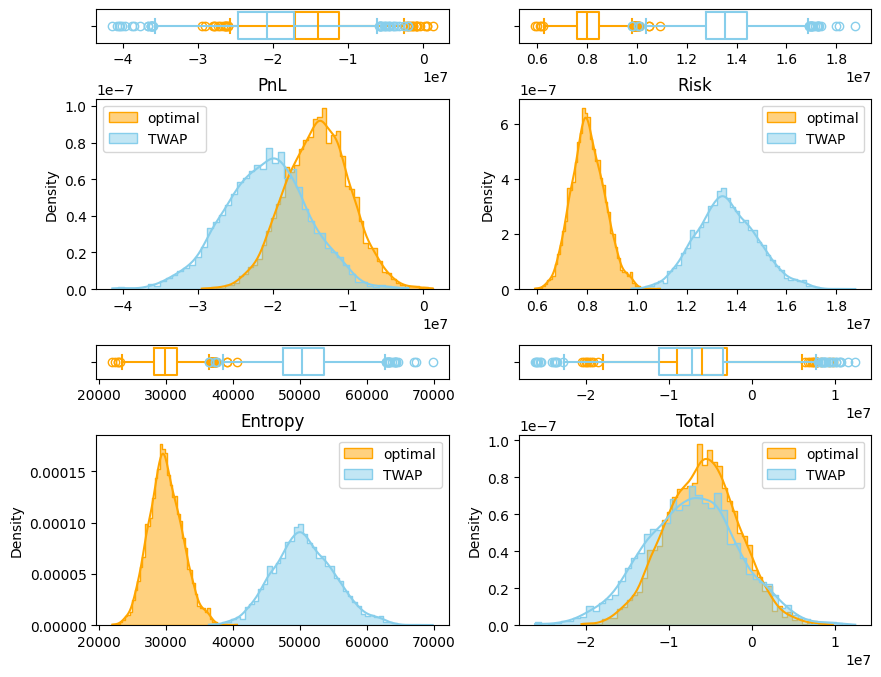}
\caption{Performance under relatively large $\gamma_M=10^{-5}$ in Model 1.}
\label{model1_value_large_gammaM}
\end{figure}

\begin{figure}
\centering
\includegraphics[width=.8\linewidth]{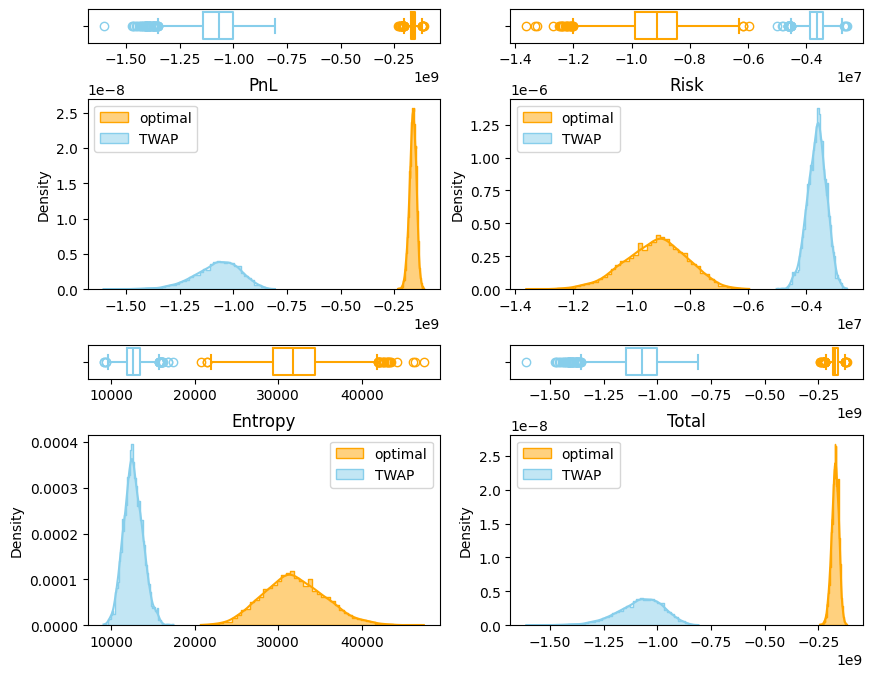}
\caption{Performance under relatively large $\eta = 10^{-3}$ and small $\delta = 2\times 10^{-4}$ in Model 1.}
\label{model1_value_large_eta_small_delta}
\end{figure}

\begin{figure}
\centering
\includegraphics[width=.8\linewidth]{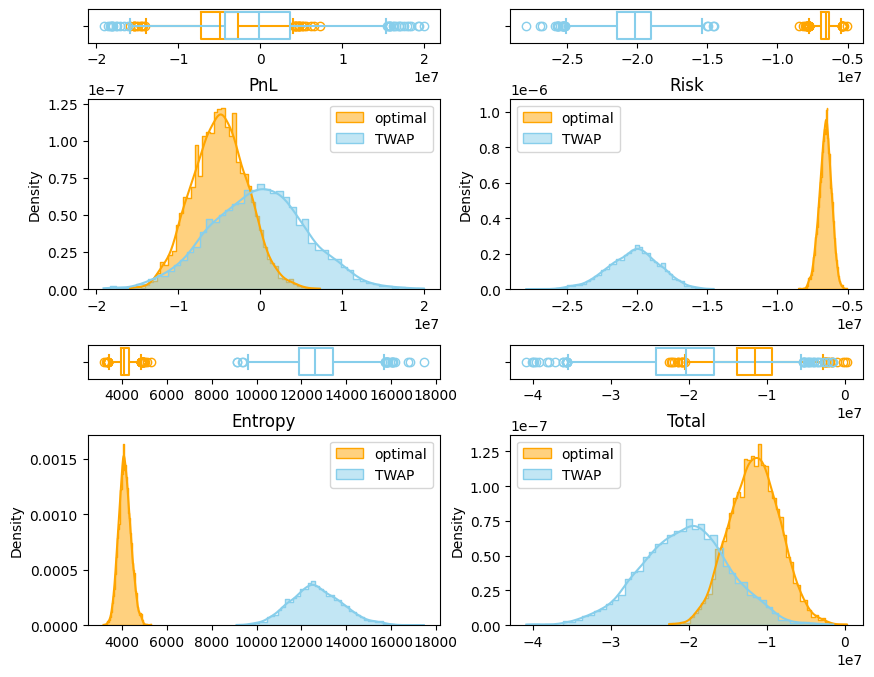}
\caption{Performance under relatively large negative $R_{xx} = -10^{-4}$ in Model 1.}
\label{model1_value_large_Rxx}
\end{figure}

\subsection{Model 2}
For Model 2, the two components $V_{PnL}$ and $V_{entropy}$ of the performance criterion $V$ remain the same as in Model 1, whereas the component $V_{risk}$ in this case becomes
\beaa
&& V_{risk} = \int_0^T \left[\frac{1}{2}R_{vv} v_t^2 + R_{va} v_t\inn{a}_t + \frac{1}{2} R_{aa} \left(\inn{a}_t^2 + \frac{1}{s_t + \beta R_{aa}}\right) \right] dt.
\eeaa
The strategies implemented in Model 2 are
\begin{enumerate}
\item Optimal strategy
\beaa
&& v_t = -\hat{A}_2\coth\left\{ \hat{A}_2(T-t) + \alpha_2 \right\}x_t, \\
&& \langle a \rangle_t = \left[ C \hat{A}_2 \coth\left\{ \hat{A}_2 (T-t) + \alpha_2 \right\} - A_2 \right] \frac{x_t}{\gamma_M};
\eeaa


\item TWAP strategy 
\beaa
&& v_t = -\frac{x_t}{T-t}, \qquad
\langle a \rangle_t = \left[ \frac{C}{T-t} - A_2 \right]\frac{x_t}{\gamma_M}.
\eeaa
\end{enumerate}
As for Model 1, in the simulations that follow the parameters below are chosen and fixed across simulations. 
\beaa 
&& \gamma = 2.5\times 10^{-7}, \; R_{va} = 5\times 10^{-6}, \; R_{aa} = 9 \times 10^{-7}, \; \beta=1, m_t = 0, \\
&& s_t = 10^{-8}, \; X=10^6, \; S=100, \; \sigma^X = 10^5, \; \sigma^S = 10, \; \rho=0.3, \; T=1.
\eeaa
We demonstrate the results in Figures \ref{model2_value_large_gammaM}, \ref{model2_value_large_eta_small_delta}, and \ref{model2_value_large_Rvv}. Each figure exhibits histograms for performance criterion and its three components of the optimal and its related TWAP strategies, topped with Box plots. The parameters as benchmark are set by
\[
\gamma_M = 2.5\times 10^{-6}, \; \eta = 2.5\times 10^{-6}, \;  \delta = 1.25\times 10^{-4}, R_{vv} = -10^{-6}
\]
whose simulation results are shown in Figure \ref{model2_value_benchmark}. 
Figure \ref{model2_value_large_gammaM} shows the result of stress testing the optimal and TWAP strategies against large permanent impact from market $\gamma_M = 10^{-5}$. Figure \ref{model2_value_large_eta_small_delta} shows the result of stress testing against relatively large $\eta=10^{-3}$ and small $\delta=2\times10^{-4}$. Finally, Figure \ref{model2_value_large_Rvv} is for stress testing against relatively large $R_{vv}=10^{-4}$. Table \ref{tab:table2} summarizes the parameters in each case and their corresponding figures.
\begin{table}[h!]
\begin{center}
    \caption{Parameters and figures for Model 2.}
    \label{tab:table2}
\begin{tabular}{|c|c|c|c|c|} 
\hline
     & $\gamma_M$ & $\eta$ & $\delta$ & $R_{vv}$  \\
\hline \hline
Benchmark & $2.5\times10^{-6}$ & $2.5\times 10^{-6}$ & $1.25\times 10^{-4}$ & $-10^{-6}$ \\
\hline
Figure \ref{model2_value_large_gammaM} & $10^{-5}$ & $2.5\times 10^{-6}$ & $1.25\times 10^{-4}$ & $-10^{-6}$ \\
\hline
Figure \ref{model2_value_large_eta_small_delta} & $2.5\times 10^{-6}$ & $10^{-3}$ & $2\times 10^{-4}$ & $-10^{-6}$ \\
\hline
Figure \ref{model2_value_large_Rvv} & $2.5\times10^{-6}$ & $2.5\times 10^{-6}$ & $1.25\times 10^{-4}$ & $-10^{-4}$ \\
\hline
\end{tabular}
\end{center}
\end{table}

We again summarize that, from the simulation results, the variabilities of the parameters $\gamma_M$, $\gamma$, $\eta$, and $\delta$ mainly contribute to the variance in P\&L, while $R_{vv}$ mainly to variance of risk. We remark that, among the histograms shown, histograms under optimal strategies shown in Figures \ref{model2_value_large_gammaM} and \ref{model2_value_large_eta_small_delta} are much more concentrated than those of TWAP, when compared with the other two cases. Notice that these two cases are for stress tests with large impact coefficients: Figure \ref{model2_value_large_gammaM} for large market permanent impact coefficient $\gamma_M$ and Figure \ref{model2_value_large_eta_small_delta} for large temporary impact $\eta$.  Heuristics can be that, as mentioned earlier that TWAP implementation does not take dynamic of market environment into account, when exposed to extreme market conditions, it performs much more wildly compared to the optimal. Overall, the optimal strategies outperform, in some cases much better, and incurring lower variance as opposed to their TWAP counterparties as shown in the figures.

\begin{figure}
\centering
\includegraphics[width=.8\linewidth]{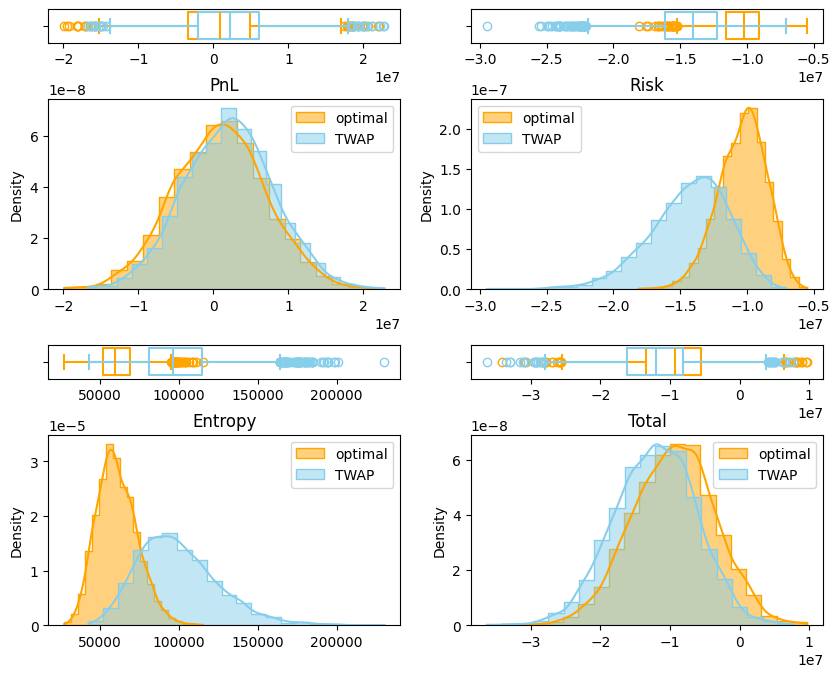}
\caption{Performance under benchmark parameters in Model 2.}
\label{model2_value_benchmark}
\end{figure}

\begin{figure}
\centering
\includegraphics[width=.8\linewidth]{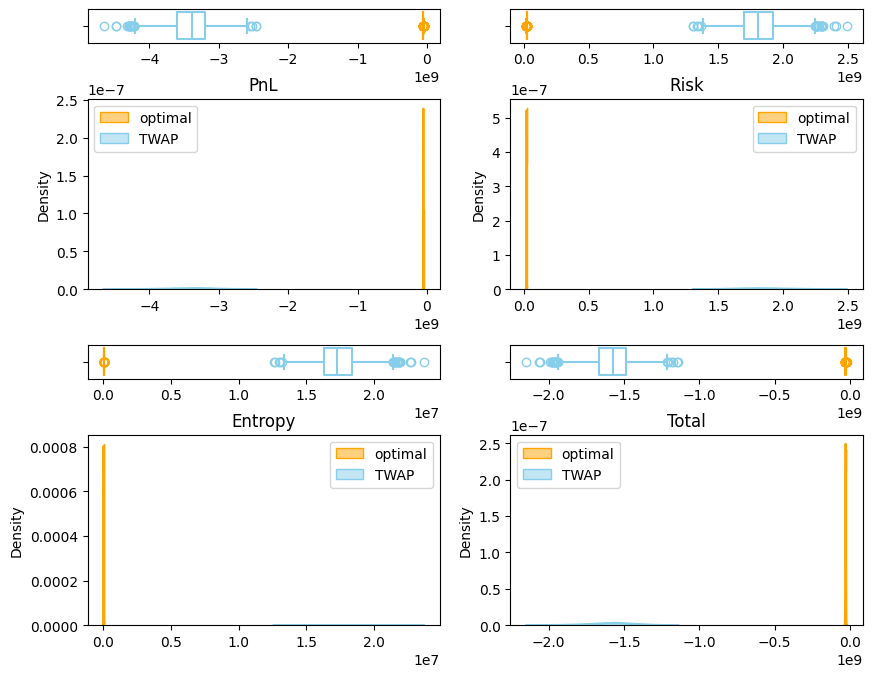}
\caption{Performance under relatively large $\gamma_M = 10^{-5}$ in Model 2.}
\label{model2_value_large_gammaM}
\end{figure}

\begin{figure}
\centering
\includegraphics[width=.8\linewidth]{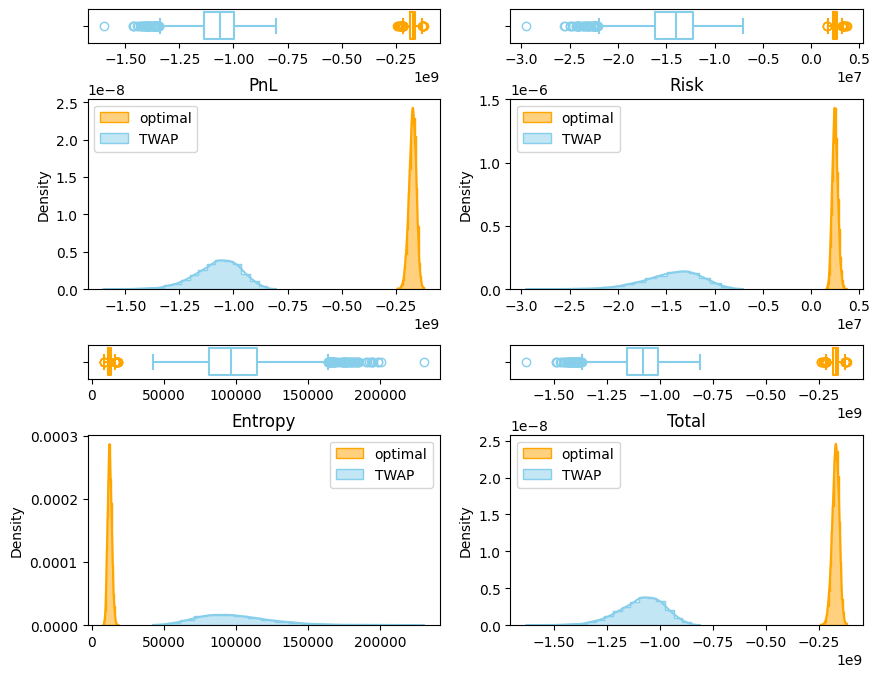}
\caption{Performance under relatively large $\eta = 10^{-3}$ and small $\delta = 2\times 10^{-4}$ in Model 2.}
\label{model2_value_large_eta_small_delta}
\end{figure}

\begin{figure}
\centering
\includegraphics[width=.8\linewidth]{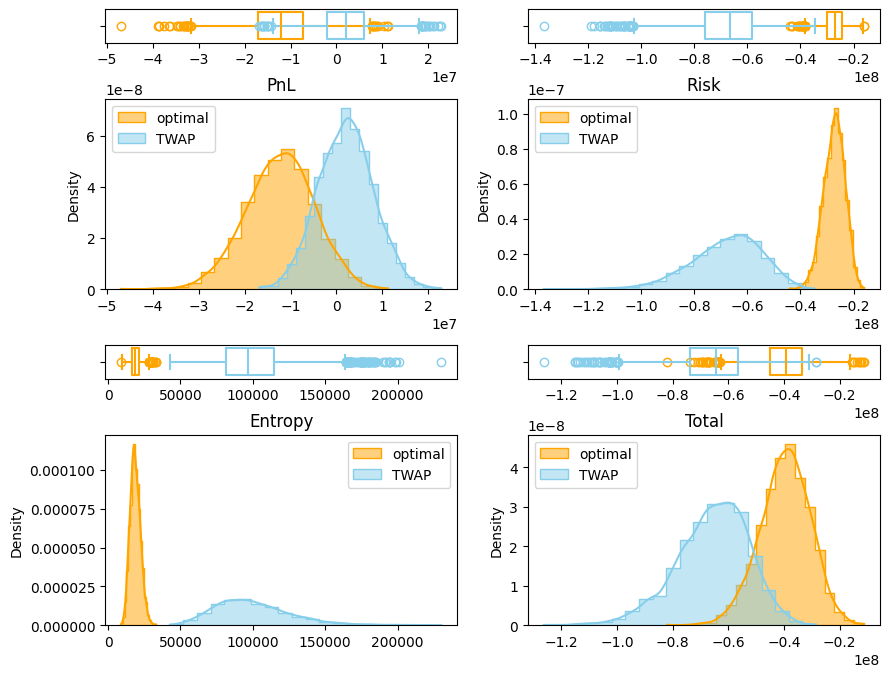}
\caption{Performance under relatively large negative $R_{vv} = -10^{-4}$ in Model 2.}
\label{model2_value_large_Rvv}
\end{figure}

\section{Conclusion and discussion} \label{sec:conclusion}
In this article, we cast the order execution problem as a relative entropy-regularized robust optimal control problem based on principle of least relative entropy for market liquidity and uncertainty. The market impact model proposed in the article added a permanent impact component from averaged market trading rate for the purpose of assessing the market's liquidity and uncertainty risk in order execution. 
Under the assumptions of Gaussian distribution and linear-quadratic framework, the value function, the optimal strategy as well as the posterior distribution for market trading rate for the resulting regularized stochastic differential game were obtainable subject to solving a system of Riccati and linear differential equations. Two models assumptions were proposed and the resulting Riccati equations and their corresponding solutions were shown. 
Calibration of the model to market data and significance of the market trading rate component worth following up in future studies. Extensions to the current model include (a) agent's impact being transient while market's impact remains permanent; (b) the use of $f$-divergence for regularization.

\section*{Acknowledgement}
We are grateful for comments from the Mathematical Finance Seminar participants at Ritsumeikan University and the participants of the XXIV Quantitative Finance Workshop at University of Cassino in Gaeta, Italy.
We thank Jiro Akahori,  Arturo Kohatsu-Higa, and Claudio Tebaldi for helpful discussions and suggestions.
THW is partially supported by the National Natural Science Foundation of China Grant 11971040.


\section{Appendix - Relative entropy-regularized control problem} \label{sec:appendix}
In this appendix, we briefly review the relative entropy-regularized control problem for reader's convenience. For more detailed discussions on entropy-regularized stochastic control problem we refer to \cite{flemming-nisio} and more recently \cite{kim-yang}, \cite{wzz}, and the references therein. 

The goal of relative entropy-regularized control problem is to determine an optimal distribution of controls that maximizes or minimizes the objective functional. Specifically,  let $\pi_t^0(a|x)$ be the ``prior" distribution of controls at time $t$ and state $x$. For any given distribution $\pi_t(a)$ of control $a$ at time $t$, consider the controlled SDE
\begin{eqnarray*}
dX_t &=& \bar\mu(X_t, \pi_t) dt + \bar\sigma(X_t, \pi_t) \, dW_t, 
\end{eqnarray*}
where
\begin{eqnarray*}
&& \bar\mu(x, \pi) = \int \mu(x, a) d\pi(a), \qquad \quad
\bar\sigma(x, \pi) = \sqrt{\int \sigma^2(x, a) d\pi(a)}
\end{eqnarray*}
for some drift $\mu$ and volatility $\sigma$. The relative entropy-regularized control problem seeks to determine an optimal ``posterior" distribution $\pi_t(a)$ of controls at time $t$ amongst certain admissible distributions $\cA$ that optimizes the following expected relative entropy-regularized, or Kullback-Leibler divergence regularized, objective functional
\begin{equation}
\min_{\pi\in\cA} \Eof{g(X_T)  + \int_0^T \int \left\{R(t, X_t, a) + \frac1\beta \log\frac{\pi_t(a)}{\pi^0_t(a|X_t)} \right\} \pi_t(a)da \, dt}  \label{eqn:r-control}
\end{equation}
subject to $\pi_t$ being a probability measure for all $t$, i.e., $\pi_t \geq 0$ and $\int \pi_t(a) da = 1$ for all $t$. $\beta > 0$ is the parameter that enforces the closeness in the sense of Kullback-Leibler divergence between the distributions $\pi_t$ and $\pi^0_t$; the larger the $\beta$, the less tightness between the two distributions is allowed. 

To solve the problem, as in the classical control theory, define the value function $V$ for the relative entropy-regularized control problem \eqref{eqn:r-control} as
\[
V(t, x) = \min_{\pi\in\cA_t}\Eof{\left. g(X_T) + \int_t^T \int \left\{ R(s, X_s, a) + \frac1\beta \log\frac{\pi_s(a)}{\pi^0_t(a|X_s)} \right\}\pi_s(a)da\,ds\right|X_t = x}. 
\]
One can show, by applying the standard argument as in the classical control theory, that $V$ satisfies the following {\it relative entropy-regularized HJB (rHJB hereafter) equation}.
\begin{equation}
\min_{\pi_t} \int \left\{V_t + \frac{\sigma^2}2V_{xx} + \mu V_x + R(t, x, a) + 
\frac1\beta\log\frac{\pi_t(a)}{\pi^0_t(a|x)} \right\} \pi_t(a) da = 0, \label{eqn:rHJB}
\end{equation}
where $\pi_t \geq 0$ and $\int \pi_t(a|x) da = 1$ for all $x$, with terminal condition $V(T, x) = g(x)$.

First order criterion by variational calculus implies that the optimal ``posterior" distribution $\pi_t$ is given in terms of feedback control as  
\begin{equation}
\pi_t(a) = \frac1{Z(x)} \pi^0_t(a|x) e^{-\beta \left\{V_t + \frac{\sigma^2}2V_{xx} + \mu V_x + R(t, x, a) \right\}}, \label{eqn:optimal-pi}
\end{equation}
where $Z(x)$ is the normalizing constant
\begin{equation}
Z(x) = \int \pi^0_t(a|x) e^{-\beta \left\{V_t + \frac{\sigma^2}2V_{xx} + \mu V_x + R(t, x, a) \right\}} da. \label{eqn:Z}
\end{equation}
We remark that \eqref{eqn:optimal-pi} resembles a Gibbs measure where $\beta$ plays the role of inverse temperature. 

By substituting \eqref{eqn:optimal-pi} and \eqref{eqn:Z} into \eqref{eqn:rHJB}, the rHJB equation \eqref{eqn:rHJB} reduces to 
\begin{equation}
-\frac1\beta \log \int \pi^0_t(a|x) e^{-\beta \left\{V_t + \frac{\sigma^2}2V_{xx} + \mu V_x + R(t, x, a) \right\}} da = 0 \label{eqn:rHJB-reduced}
\end{equation}
or equivalently
\begin{equation}
\int \pi^0_t(a|x) e^{-\beta \left\{V_t + \frac{\sigma^2}2V_{xx} + \mu V_x + R(t, x, a) \right\}} da = 1 \label{eqn:rHJB-reduced-1}
\end{equation}
with terminal condition $V(T, x) = g(x)$. We conclude that the optimal ``posterior" distribution $\pi^*_t$ is given in terms of the value function $V$ and the ``prior" distribution $\pi^0_t$ as 
\[
\pi^*_t(a) = \pi^0_t(a|x) e^{-\beta \left\{V_t + \frac{\sigma^2}2V_{xx} + \mu V_x + R(t, x, a) \right\}}
\]
and the rHJB equation \eqref{eqn:rHJB-reduced-1} guarantees that the $\pi_t^*$ given above is indeed a probability distribution. 

As an extension to the relative entropy-regularized control problem, one may consider the following {\it $f$-divergence regularized} control problem.  
\begin{equation}
\min_{\pi} \Eof{g(X_T)  + \int_0^T \int \left\{R(t, X_t, a) + \frac1\beta f\left(\frac{\pi_t(a)}{\pi^0_t(a)}\right) \right\} \pi_t(a) da \, dt} \label{f-control}
\end{equation}
for some function $f$. We remark that the Kullback-Leibler divergence is recovered by letting $f(x) = \log x$ and the Tsallis divergence by $f(x) =  \frac1{p - 1} (x^{p - 1} - 1)$ for $p > 1$.

By the same token, one can show that the value of the control problem \eqref{f-control} satisfies an rHJB equation. Let $\tilde f(x) = x f(x)$, $\cL V = \frac{\sigma^2}2V_{xx} + \mu V_x + R(t, x, a)$ and $\lambda$ the Lagrange multiplier. The $f$-divergence rHJB equation reads
\begin{eqnarray*}
&& \int \left\{V_t + \cL V + 
\frac1\beta f\circ\left(\tilde f'\right)^{-1}\left(-\beta\cL V + \beta\lambda\right) \right\}
\left(\tilde f'\right)^{-1}(-\beta\cL V + \beta\lambda) \pi^0_t(a|x) da = 0, \\
&& \int \left(\tilde f'\right)^{-1}(-\beta\cL V + \beta\lambda) \pi^0_t(a|x) da = 1 
\end{eqnarray*}
which unfortunately cannot be further simplified. We note that, in the case of relative entropy, i.e., $f(x) = \log x$, the two equations above can be combined into one, as shown in \eqref{eqn:rHJB-reduced}. 

\section*{Conflicts of Interest} 
The authors declare no conflict of interest.


\end{document}